\RequirePackage{fix-cm}
\documentclass[smallextended,envcountsect,envcountsame]{svjour3}

\smartqed

\usepackage[utf8]{inputenc} 
\usepackage[T1]{fontenc}    
\usepackage{url}            
\usepackage{booktabs}       
\usepackage{algorithm}
\usepackage{algpseudocode}
\usepackage{algorithmicx}
\usepackage{amsfonts}       
\usepackage{nicefrac}       
\usepackage{microtype}      
\usepackage{xcolor}         
\usepackage{graphicx}
\usepackage{tikz}
\usepackage{tikz-3dplot}
\usetikzlibrary{calc,intersections,patterns,decorations.pathreplacing,shapes.geometric,arrows.meta}

\usepackage[numbers,sort&compress]{natbib}
\usepackage{amsmath,amssymb}
\usepackage[colorlinks=true,allcolors=blue]{hyperref}
\usepackage[capitalize,noabbrev]{cleveref}

\newcommand{\T}{\mathsf{T}}

\newcommand{\set}[1]{\left\{#1\right\}} 
 
\newcommand{\R}{\mathbb{R}}

\newcommand{\Z}{\mathbb{Z}}
\newcommand{\x}{\mathbf{x}}
\newcommand{\y}{\mathbf{y}}
\newcommand{\w}{\mathbf{w}}
\newcommand{\z}{\mathbf{z}}

\newcommand{\q}{\mathbf{q}}

\renewcommand{\u}{\mathbf{u}}
\renewcommand{\v}{\mathbf{v}}
\renewcommand{\a}{\mathbf{a}}

\renewcommand{\c}{\mathbf{c}}

\renewcommand{\H}{\mathcal{H}}

\newcommand{\0}{\mathbf{0}}

\DeclareMathOperator{\argmax}{arg\,max} 
\DeclareMathOperator{\intset}{int} 

\DeclareMathOperator{\proj}{proj} 

\newcommand{\dd}{\,\mathrm{d}}

\usepackage{bm}
\newcommand{\1}{\mathbf{1}} 

\definecolor{figblue}{RGB}{0,45,114}
\definecolor{figred}{RGB}{191,47,56}
\definecolor{figgold}{RGB}{218,165,32}
\definecolor{figgreen}{RGB}{34,139,34}

\newcommand{\vol}{\operatorname{vol}}

\journalname{}
\date{}
\makeatletter\def\makeheadbox{}\makeatother

\begin{document}

\title{Linear Threshold for Oertel's Conjecture on the Mixed-Integer Volume}
\titlerunning{Linear Threshold for Oertel's Conjecture}

\author{Hongyu Cheng \and Amitabh Basu}
\institute{Department of Applied Mathematics and Statistics, Johns Hopkins University, Baltimore, MD 21218, USA\\
\email{hongyucheng@jhu.edu, basu.amitabh@jhu.edu}}

\maketitle

\begin{abstract}
Gr{\"u}nbaum's inequality guarantees that the centroid of a convex body has halfspace depth at least $1/e$: every halfspace containing the centroid captures at least a $1/e$ fraction of the body's volume. For mixed-integer convex sets $S=C\cap\left(\Z^n\times\R^d\right)$ where $C$ is a convex body, \citet{oertel2014integer} conjectured that there exists $\y\in S$ such that every closed halfspace $H$ containing $\y$ satisfies $\H_d(S\cap H)\ge \frac{1}{2^n e}\H_d(S)$, where $\H_d(X)$ denotes the $d$-dimensional Hausdorff measure of $X\subseteq \Z^n\times\R^d$. This conjecture is closely connected to complexity bounds for cutting plane methods and information complexity in mixed-integer convex optimization. \citet{basu2017centerpoints} established this conjecture for sets of sufficiently large lattice width, with the required lower bound on lattice width depending exponentially on the dimension. More recently, \citet{cristi2025reducing} reduced this threshold to a polynomial one by assuming that $\proj_{\R^n}(C)$ contains a Euclidean ball of radius at least $1178\, d^2 n^{3/2}$. We prove that when $\proj_{\R^n}(C)$ contains an $\ell_\infty$ ball of radius $k\ge \frac{3e}{2}(n+d)$, which is linear in the dimensions $n$ and $d$, there exists $\y^\star\in S$ such that every closed halfspace $H$ containing $\y^\star$ satisfies
	\[
		\H_d(S\cap H)\ge \left(\frac{1}{e}-\frac{3(n+d)}{2k}\right)\H_d(S).
	\]
In particular, when $k\ge 3e(n+d)$ we obtain $\H_d(S\cap H)\ge \frac{1}{2e}\H_d(S)\ge \frac{1}{2^n e}\H_d(S)$, thus verifying Oertel's conjecture for a significantly larger family of sets than previous results. We also show that this linear scaling is necessary: when the radius is sublinear in the total dimension, the maximum achievable halfspace depth can be arbitrarily small relative to the total mixed-integer volume, so no dimension-independent constant fraction lower bound is possible under such an assumption alone. However, the conjecture remains open in its full generality.

\keywords{Mixed-integer convex optimization \and halfspace depth \and centerpoints \and cutting plane methods \and Gr{\"u}nbaum inequality}
\subclass{90C10 \and 90C25 \and 52A40}
\end{abstract}

\section{Introduction}

Cutting plane methods for convex minimization maintain a localization region and repeatedly refine it using separating halfspaces~\citep{kelley1960cutting,nemirovski1983problem}. Many progress bounds for such methods are stated in terms of volume decrease or related volume based potentials~\citep{vaidya1996new}. This motivates geometric depth guarantees: a halfspace containing a specified point should capture a constant fraction of the volume. Grünbaum's inequality provides such a guarantee for the centroid: any hyperplane through the centroid of a convex body splits it so that each closed side contains at least a $1/e$ fraction of the volume~\citep{grunbaum1960partitions}.

We study a mixed-integer generalization of this depth guarantee. Let $C\subset \R^{n+d}$ be a convex body and let $S=C\cap(\Z^n\times\R^d)$. Following~\citet{oertel2014integer}, we measure the size of $S$ by its mixed-integer volume $\H_d(S)$, the $d$-dimensional Hausdorff measure of $S$ as a subset of the mixed-integer lattice $\Z^n\times\R^d$; or equivalently, the sum $\sum_{\z\in\Z^n}\vol_d\bigl(S\cap(\{\z\}\times\R^d)\bigr)$, where $\vol_d(\cdot)$ denotes the usual $d$-dimensional volume. Oertel conjectured that there exists a point $\y\in S$ such that every closed halfspace $H$ containing $\y$ satisfies $\H_d(S\cap H)\ge \frac{1}{2^n e}\H_d(S)$. The factor $2^n$ is unavoidable in general. In the pure integer case $d=0$, taking $C=[0,1]^n$ gives $S=\{0,1\}^n$, and for every $\y\in S$ there exists a closed halfspace containing $\y$ that captures only one point of $S$, so $\H_0(S\cap H)\le 2^{-n}\H_0(S)$. This scale is also consistent with the combinatorial complexity of the integer lattice: Doignon's theorem shows that the Helly number of $\Z^n$ equals $2^n$~\citep{doignon1973convexity}, and related Helly type bounds appear throughout integer optimization~\citep{averkov2013sfree}. Oertel's conjecture is closely connected to complexity bounds in cutting plane methods and to the information complexity of mixed-integer convex optimization~\citep{basu2017centerpoints,basu2024infocomp}.

Existing results verify the conjecture (stated precisely below as Conjecture~\ref{conj:oertel}) for sets whose projection $K=\proj_{\R^n}(C)$ is large in a quantitative sense. \citet{basu2017centerpoints} established it when $K$ has lattice width at least $\Omega((n+d)^{5/2}n^{n+2})$, a threshold that grows exponentially in $n$. \citet{cristi2025reducing} reduced this to a polynomial threshold by assuming that $K$ contains a Euclidean ball of radius at least $1178\, d^2 n^{3/2}$.

Our main result, Theorem~\ref{thm:main}, establishes a linear threshold. If $K$ contains a translate of $k\mathbb{B}_\infty^n$ with $k\ge \frac{3e}{2}(n+d)$, then there exists $\y^\star\in S$ such that every closed halfspace $H$ containing $\y^\star$ satisfies
\[
	\H_d(S\cap H)\ge \Bigl(\frac{1}{e}-\frac{3(n+d)}{2k}\Bigr)\H_d(S).
\]
When $k\ge 3e(n+d)$, the bound becomes $\H_d(S\cap H)\ge \frac{1}{2e}\H_d(S)$, which implies the conjectured $\frac{1}{2^n e}$ bound for all $n\ge 1$. Corollary~\ref{cor:main} states the corresponding centerpoint guarantee. We also show that a linear threshold is necessary in a precise quantitative sense: Theorem~\ref{thm:linear-necessity-total-dim} and Corollary~\ref{cor:linear-necessity-sublinear} prove that any dimension independent constant fraction guarantee based solely on the existence of a translate of $k\mathbb{B}_\infty^n$ in $K$ requires $k$ to scale linearly in $n+d$.

We compare our $\ell_\infty$ ball condition with the Euclidean ball condition in~\citet{cristi2025reducing}. A Euclidean ball of radius $R$ contains the $\ell_\infty$ ball $(R/\sqrt{n})\mathbb{B}_\infty^n$. Hence, $R\ge \frac{3e}{2}\sqrt{n}\,(n+d)$ suffices to apply Theorem~\ref{thm:main}, and $R\ge 3e\sqrt{n}\,(n+d)$ suffices for the $\frac{1}{2e}$-depth guarantee in Corollary~\ref{cor:main}. Compared to the threshold $R\ge 1178\, d^2 n^{3/2}$ in~\citet{cristi2025reducing}, our Corollary threshold replaces the scaling $d^2 n^{3/2}$ by $\sqrt{n}(n+d)$, which corresponds to dropping the sum of exponents from $3.5$ to $1.5$. It also lowers the leading constant from $1178$ to $3e\approx 8.15$ and replaces quadratic dependence on $d$ by linear dependence. At a high level, our proof compares mixed-integer volume and Lebesgue volume via Minkowski thickening and erosion (Lemma~\ref{lem:integral-thick}). This comparison exploits slice concavity from the Brunn--Minkowski theory~\citep{gardner2002brunn,schneider2014convex} together with a symmetrization argument~\citep{kesavan2006symmetrization}. We then perform a centroid rounding step and conclude using Grünbaum's inequality~\citep{grunbaum1960partitions}.

The remainder of the paper is organized as follows. Section~\ref{sec:notation} fixes notation and states our main results. Section~\ref{sec:comparison} develops the volume comparison tools. Section~\ref{sec:proof} proves the main theorem and corollary. Section~\ref{sec:necessity} establishes the necessity of a linear threshold.

\section{Background and main results}\label{sec:notation}

\subsection{Basic notation}

Throughout, $n,d\in \Z_{>0}$. We write points in $\R^{n+d}=\R^n\times \R^d$ as
\[
	\y=(\z,\x),\qquad \z\in \R^n,\ \x\in \R^d,
\]
and let $\proj_{\R^n}:\R^{n+d}\to \R^n$ be the projection onto the first factor, i.e., $\proj_{\R^n}(\z,\x)=\z$.
For a set $A\subseteq \R^{n+d}$, we write
\[
	\proj_{\R^n}(A):=\set{\proj_{\R^n}(\y):\ \y\in A}\subseteq \R^n.
\]
For $\z\in \R^n$, we define the $\z$-fiber of $A$ by
\[
	A_{\z}:=\set{\x\in \R^d:\ (\z,\x)\in A}\subseteq \R^d.
\]

For the remainder of this paper, fix a convex body $C\subset \R^{n+d}$, i.e., a compact convex set with $\intset(C)\neq \emptyset$, and define the mixed-integer lattice and the feasible mixed-integer set
\[
	\Lambda:=\Z^n\times \R^d,\qquad S:=C\cap \Lambda,\qquad K:=\proj_{\R^n}(C)\subset \R^n.
\]
For any convex body $D\subset \R^{n+d}$, we write
\[
	S_D:=D\cap \Lambda,\qquad K_D:=\proj_{\R^n}(D)\subset \R^n.
\]

We will use the cube
\[
	Q:=[-1/2,1/2]^n,\qquad \mathbb{B}_\infty^n:=\set{\u\in \R^n:\ \|\u\|_\infty\le 1},
\]
so that $\mathbb{B}_\infty^n=2Q$.
For $A\subseteq \R^n$, define the \emph{Minkowski sum} and \emph{erosion} by
\[
	A+Q:=\set{a+\q:\ a\in A,\ \q\in Q},\qquad A\ominus Q:=\set{\z\in \R^n:\ \z+Q\subseteq A}.
\]

\subsection{Mixed-integer volume}

For $m\in \Z_{>0}$, we write $\vol_m$ for the $m$ dimensional Lebesgue measure (with the ambient space clear from context).

For a measurable set $A\subseteq \Lambda$, we define its \emph{mixed-integer volume} as
\begin{equation}\label{eq:Hd-fiber-sum}
	\H_d(A):=\sum_{\z\in \Z^n}\vol_d\bigl(A_{\z}\bigr).
\end{equation}

If $A$ is bounded, then only finitely many $\z\in \Z^n$ yield a nonempty fiber, and hence $\H_d(A)<\infty$.
In particular, $\H_d(S)<\infty$ since $C$ is bounded.

\subsection{Halfspace depth and centerpoints}

A (closed) halfspace in $\R^{n+d}$ is a set of the form
\[
	H=\set{\y\in \R^{n+d}:\ \a^\T \y \ge c},\qquad \a\in \R^{n+d}\setminus\{\0\},\ c\in \R.
\]

\begin{definition}[Halfspace depth and centerpoints]\label{def:depth}
	For $\y\in S$, define the halfspace depth of $\y$ with respect to $S$ by
	\[
		h_S(\y):=\inf\set{\H_d(S\cap H):\ H\subset \R^{n+d} \text{ is a closed halfspace and } \y\in H}.
	\]
	The set of centerpoints of $S$ is
	\[
		\mathcal{C}(S):=\argmax_{\y\in S} h_S(\y).
	\]
\end{definition}

For $m\in \Z_{>0}$ and a convex body $D\subset \R^m$, we denote its centroid by
\[
	\mathfrak{c}(D):=\frac{1}{\vol_m(D)}\int_D \u\dd \u.
\]

\noindent
We will use the classical Grünbaum inequality~\citep{grunbaum1960partitions} in the following form.

\begin{theorem}[Grünbaum]\label{thm:grunbaum}
	Let $m\in \Z_{>0}$ and let $D\subset \R^m$ be a convex body with centroid $\mathfrak{c}(D)$.
	For every closed halfspace $H\subset \R^m$ with $\mathfrak{c}(D)\in H$,
	\[
		\vol_m(D\cap H)\ \ge\ \Bigl(\frac{m}{m+1}\Bigr)^m \vol_m(D)\ \ge\ \frac{1}{e}\vol_m(D).
	\]
\end{theorem}

\citet{oertel2014integer} conjectured the following mixed-integer analogue, with the additional factor $1/2^n$ reflecting the Helly number of $\Z^n$.

\begin{conj}[Oertel's conjecture]\label{conj:oertel}
	For any convex body $C\subset \R^{n+d}$, there exists a point $\y\in S=C\cap(\Z^n\times\R^d)$ such that for every closed halfspace $H$ containing $\y$,
	\[
		\H_d(S\cap H)\ge \frac{1}{2^n e}\H_d(S).
	\]
\end{conj}

\subsection{Main results}

Our main result verifies Conjecture~\ref{conj:oertel} under a linear threshold condition on the $\ell_\infty$ radius of the projection.

\begin{theorem}\label{thm:main}
	Let $C\subset \R^{n+d}$ be a convex body, and let $S=C\cap\left(\Z^n\times \R^d\right)$. Write $K=\proj_{\R^n}(C)\subset \R^n$.
	Assume that there exist $\z_0\in \R^n$ and $k>0$ such that
	\begin{equation}\label{eq:big-box}
		\z_0+k\mathbb{B}_\infty^n \subseteq K,\qquad k\ge \frac{3e}{2}(n+d).
	\end{equation}
	Then there exists $\y^\star\in S$ such that for every closed halfspace $H\subseteq \R^{n+d}$ with $\y^\star\in H$,
	\begin{equation}\label{eq:main-bound}
		\H_d(S\cap H)\ \ge\ \Bigl(\frac{1}{e}-\frac{3(n+d)}{2k}\Bigr)\H_d(S).
	\end{equation}
\end{theorem}

If in addition $k\ge 3e(n+d)$, then the coefficient in \eqref{eq:main-bound} satisfies $\frac{1}{e}-\frac{3(n+d)}{2k}\ge \frac{1}{2e}$, and hence \eqref{eq:main-bound} yields $\H_d(S\cap H)\ge \frac{1}{2e}\,\H_d(S)$ for every closed halfspace $H\subseteq \R^{n+d}$ with $\y^\star\in H$.
As $\frac{1}{2e}\ge \frac{1}{2^n e}$ for all $n\ge 1$, this verifies Conjecture~\ref{conj:oertel} for all sets satisfying \eqref{eq:big-box} with $k\ge 3e(n+d)$.
We summarize the corresponding centerpoint implication in Corollary~\ref{cor:main}.

\begin{corollary}\label{cor:main}
	Assume the hypotheses of Theorem~\ref{thm:main}, and suppose in addition that $k\ge 3e(n+d)$.
	Then $\mathcal{C}(S)\neq \emptyset$, and every $\hat \y\in \mathcal{C}(S)$ satisfies
	\begin{equation}\label{eq:centerpoint-bound}
		h_S(\hat \y)\ \ge\ \frac{1}{2e}\,\H_d(S)\ \ge\ \frac{1}{2^n e}\,\H_d(S).
	\end{equation}
\end{corollary}

We also show that a linear $\ell_\infty$ radius in the total dimension $n+d$ is necessary if one seeks a dimension-free constant fraction guarantee based only on the existence of a translate of $k\mathbb{B}_\infty^n$ in the projection.
See Theorem~\ref{thm:linear-necessity-total-dim} and Corollary~\ref{cor:linear-necessity-sublinear}.

Theorem~\ref{thm:main} remains valid if the axis aligned cube $k\mathbb{B}_\infty^n$ is replaced by a unimodular image $kU\mathbb{B}_\infty^n$. This relaxes \eqref{eq:big-box} by allowing the cube to be aligned with an arbitrary integer basis of $\Z^n$.
\begin{corollary}\label{cor:unimodular}
	Let $C\subset \R^{n+d}$ be a convex body with $S:=C\cap\left(\Z^n\times \R^d\right)$ and $K:=\proj_{\R^n}(C)$.
	Assume that there exist $\z_0\in \R^n$, $k\ge \frac{3e}{2}(n+d)$, and an integer matrix $U\in\Z^{n\times n}$ with $|\det U|=1$ such that
	\begin{equation}\label{eq:big-box-unimod}
		\z_0+k\,U\mathbb{B}_\infty^n\subseteq K.
	\end{equation}
	Then there exists $\y^\star\in S$ such that for every closed halfspace $H\subseteq \R^{n+d}$ with $\y^\star\in H$,
	\[
		\H_d(S\cap H)\ \ge\ \Bigl(\frac{1}{e}-\frac{3(n+d)}{2k}\Bigr)\H_d(S).
	\]
\end{corollary}

\section{A comparison lemma for mixed-integer volume}\label{sec:comparison}

The goal of this section is to prove the following comparison inequality, which relates the mixed-integer volume $\H_d(S_D)$ to the Lebesgue volume $\vol_{n+d}(D)$ when the projection of $D$ contains a large $\ell_\infty$ ball.

\begin{lemma}[Comparison inequality]\label{lem:disc-cont}
	Let $D\subset \R^{n+d}$ be a convex body and recall $K_D:=\proj_{\R^n}(D)\subset \R^n$ and $S_D:=D\cap\left(\Z^n\times \R^d\right)$.
	Assume that there exist $\z_D\in \R^n$ and $k>1/2$ such that
	\begin{equation}\label{eq:k-box}
		\z_D+k\mathbb{B}_\infty^n\subseteq K_D.
	\end{equation}
	Then we have
	\begin{equation}\label{eq:disc-cont-bound}
		\Bigl(1-\frac{1}{2k}\Bigr)^{n+d}\vol_{n+d}(D)\ \le\ \H_d(S_D)\ \le\ \Bigl(1+\frac{1}{2k}\Bigr)^{n+d}\vol_{n+d}(D).
	\end{equation}
\end{lemma}

\noindent
The key idea is to express both volumes via layer cake decompositions (see, for example, \citep[Theorem 8.16]{rudin1987real}) over the superlevel sets $L_s:=\set{\z\in\R^n:g(\z)\ge s}$, where $g(\z):=\vol_d(D_\z)^{1/d}$ is the ``slice radius'':
\[
	\H_d(S_D)=d\int_0^\infty s^{d-1}\#(L_s\cap\Z^n)\dd s,\qquad
	\vol_{n+d}(D)=d\int_0^\infty s^{d-1}\vol_n(L_s)\dd s.
\]
The second equality is the usual layer cake formula \eqref{eq:layercake-power}, and the first identity follows from \eqref{eq:Hd-fiber-sum} with the same layer cake argument for counting measure (derived in \eqref{eq:nu-layer}).
A lattice sandwich lemma (Lemma~\ref{lem:lattice-sandwich}) then bounds the lattice point count by the Lebesgue measure of thickened and eroded level sets, where we recall that $Q=[-1/2,1/2]^n$:
\[
	\vol_n(L_s\ominus Q)\ \le\ \#(L_s\cap \Z^n)\ \le\ \vol_n(L_s+Q).
\]
The main technical ingredient (Lemma~\ref{lem:integral-thick}) controls the effect of thickening and erosion on these volumes via a Brunn--Minkowski based concavity argument for $g$ and a scaling argument:
\[
	d\int_0^\infty s^{d-1}\vol_n(L_s+Q)\dd s \le \Bigl(1+\frac{1}{2k}\Bigr)^{n+d}
	d\int_0^\infty s^{d-1}\vol_n(L_s)\dd s,
\]
with an analogous lower bound for erosion. Combining these ingredients yields the comparison inequality in Lemma~\ref{lem:disc-cont}.

The remainder of this section develops these tools. Section~\ref{subsec:layer-cake} establishes the layer cake representations and the lattice sandwich lemma. Section~\ref{subsec:concavity} proves the concavity of slice radii via Brunn--Minkowski. Section~\ref{subsec:thickening} proves the integrated thickening/erosion bounds. Section~\ref{subsec:comparison-proof} assembles the proof of Lemma~\ref{lem:disc-cont}.

\subsection{Layer cake decomposition}\label{subsec:layer-cake}

\begin{lemma}[Lattice sandwich]\label{lem:lattice-sandwich}
	Let $A\subseteq \R^n$ be closed. Then
	\begin{equation}\label{eq:lattice-sandwich}
		\vol_n(A\ominus Q)\ \le\ \#(A\cap \Z^n)\ \le\ \vol_n(A+Q).
	\end{equation}
\end{lemma}
\begin{proof}
	Since $A$ is closed and $Q$ is compact, the sets $A\ominus Q$ and $A+Q$ are closed and hence Borel measurable.

	For the left inequality, consider the union
	\[
		U:=\bigcup_{\z\in A\cap \Z^n}(\z+Q).
	\]
	We claim that $A\ominus Q\subseteq U$. Fix $\zeta\in A\ominus Q$. Then $\zeta+Q\subseteq A$.
	Since $Q=[-1/2,1/2]^n$, there exists $\z\in \Z^n$ with $\z\in \zeta+Q$ (choose $\z_i$ by rounding $\zeta_i$ to the nearest integer point).
	As $\zeta+Q\subseteq A$, we obtain $\z\in A\cap \Z^n$, and also $\zeta\in \z+Q$. This shows $\zeta\in U$.

	By subadditivity of Lebesgue measure and the fact that $\vol_n(Q)=1$,
	\[
		\vol_n(A\ominus Q)\le \vol_n(U)\le \sum_{\z\in A\cap \Z^n}\vol_n(\z+Q)= \#(A\cap \Z^n).
	\]

	For the right inequality, note that for every $\z\in A\cap \Z^n$,
	\[
		\z+\intset(Q) \subseteq A+Q,
	\]
	since $\z\in A$ and $\intset(Q)\subseteq Q$. The sets $\{\z+\intset(Q):\ \z\in \Z^n\}$ are pairwise disjoint.
	Therefore,
	\[
		\#(A\cap \Z^n) = \sum_{\z\in A\cap \Z^n}\vol_n(\intset(Q)) = \vol_n\Bigl(\bigcup_{\z\in A\cap \Z^n}(\z+\intset(Q))\Bigr) \le \vol_n(A+Q),
	\]
	which completes the proof.
\end{proof}

Lemma~\ref{lem:lattice-sandwich} is the bridge between lattice point counting and continuous volume. The Minkowski sum $A+Q$ and erosion $A\ominus Q$ provide upper and lower bounds on the lattice point count in terms of Lebesgue measure. When $A$ is sufficiently large compared to $Q$, the difference between $\vol_n(A+Q)$ and $\vol_n(A\ominus Q)$ is small relative to $\vol_n(A)$, making the bounds tight.
See, for example, \citet{beck2015computing}.

\begin{lemma}[Layer cake]\label{lem:layer-cake}
	Let $g:\R^n\to [0,\infty]$ be Lebesgue measurable and let $d\in \Z_{>0}$. Then
	\begin{equation}\label{eq:layercake-power}
		\int_{\R^n} g(\z)^d\dd \z=d\int_0^\infty s^{d-1}\vol_n\bigl(\set{\z:\ g(\z)\ge s}\bigr)\dd s.
	\end{equation}
\end{lemma}
\begin{proof}
	Fix Lebesgue measurable $g\ge 0$ and an integer $d\ge 1$. For any $a\in [0,\infty]$ we have
	$a=\int_0^\infty \1_{\{a\ge t\}}\dd t$. Therefore,
	\begin{align*}
		\int_{\R^n} g(\z)^d\dd \z
		 & = \int_{\R^n}\int_0^\infty \1_{\{g(\z)^d\ge t\}}\dd t\,\dd \z       \\
		 & = \int_0^\infty \int_{\R^n}\1_{\{g(\z)\ge t^{1/d}\}}\dd \z\,\dd t   \\
		 & = \int_0^\infty \vol_n\bigl(\set{\z:\ g(\z)\ge t^{1/d}}\bigr)\dd t    \\
		 & = d\int_0^\infty s^{d-1}\vol_n\bigl(\set{\z:\ g(\z)\ge s}\bigr)\dd s,
	\end{align*}
	where the second line uses Tonelli's theorem and the last uses the substitution $t=s^d$.
\end{proof}

The power of the layer cake decomposition lies in the following observation. For a nonnegative function $g$ on $K$, both the continuous integral $\int_K g(\z)^d\dd \z$ and the discrete sum $\sum_{\z\in K\cap \Z^n} g(\z)^d$ admit parallel representations via the superlevel sets $L_s:=\{g\ge s\}$. Indeed, Lemma~\ref{lem:layer-cake} gives
\[
	\int_K g(\z)^d\dd \z = d\int_0^\infty s^{d-1} \vol_n(L_s)\dd s,
\]
and the same calculation applied to the counting measure yields (as we will derive in \eqref{eq:nu-layer})
\[
	\sum_{\z\in K\cap \Z^n} g(\z)^d = d\int_0^\infty s^{d-1} \#(L_s\cap \Z^n)\dd s.
\]
This structural similarity reduces the comparison of continuous and discrete volumes to a purely geometric problem: relating the lattice point count $\#(L_s\cap \Z^n)$ to the Lebesgue measure $\vol_n(L_s)$ for each level set $L_s$, via Lemma~\ref{lem:lattice-sandwich}.

\subsection{Concavity of slice radii}\label{subsec:concavity}

For $\z\in \R^n$, define
\begin{equation}\label{eq:f}
	f(\z):=\vol_d(C_{\z}).
\end{equation}
Since $C_{\z}=\emptyset$ for $\z\notin K$, we have $f(\z)=0$ on $\R^n\setminus K$.

\begin{lemma}[Upper semicontinuity of slice volumes]\label{lem:slice-usc}
	The function $f:\R^n\to [0,\infty)$ is upper semicontinuous. In particular, for every $s>0$ the set
	\[
		\set{\z\in \R^n:\ f(\z)\ge s}
	\]
	is closed in $\R^n$ and contained in $K$.
\end{lemma}
\begin{proof}
	Fix $\z\in \R^n$ and let $\z_k\to \z$ in $\R^n$. We show $\limsup_{k\to\infty} f(\z_k)\le f(\z)$.

	If $\z\notin K$, then $\delta:=\operatorname{dist}(\z,K)>0$ since $K$ is closed.
	For all large $k$, $\|\z_k-\z\|_2<\delta/2$, hence $\operatorname{dist}(\z_k,K)\ge \delta/2>0$ and $\z_k\notin K$.
	Therefore $f(\z_k)=0=f(\z)$ for all large $k$, so the claim follows.

	Now assume $\z\in K$. Let $\epsilon>0$ and set $U:=C_{\z}+\epsilon \set{\x \in\R^d:\ \|\x\|_2< 1}$. The set $U$ is open in $\R^d$ and contains $C_{\z}$.
	We claim that $C_{\z_k}\subseteq U$ for all large $k$ such that $\z_k\in K$.
	If not, then there exist a subsequence $\z_{k_j}$ and points $\x_{k_j}\in C_{\z_{k_j}}$ with $\x_{k_j}\notin U$.
	Since $C$ is compact, the sequence $(\z_{k_j},\x_{k_j})\in C$ has a convergent subsequence to some $(\z,\x)\in C$.
	Then $\x\in C_{\z}\subseteq U$, and since $U$ is open, $\x_{k_j}\in U$ for all large $j$, a contradiction.

	Therefore, for all large $k$, we have $C_{\z_k}\subseteq C_{\z}+\epsilon \mathbb{B}_2^d$ (if $\z_k \not\in K$, then $C_{\z_k} = \emptyset$), and thus
	\[
		f(\z_k)=\vol_d(C_{\z_k})\le \vol_d(C_{\z}+\epsilon \mathbb{B}_2^d).
	\]
	Taking $\limsup_{k\to\infty}$ gives $\limsup_{k\to\infty} f(\z_k)\le \vol_d(C_{\z}+\epsilon \mathbb{B}_2^d)$.
	Now take $\epsilon_m:=1/m$ and set $E_m:=C_{\z}+\epsilon_m \mathbb{B}_2^d$.
	Then $(E_m)$ is a decreasing sequence of closed sets with $\bigcap_{m\ge 1} E_m=C_{\z}$.
	Since $C$ is compact, $E_1$ is bounded and hence $\vol_d(E_1)<\infty$.
	By continuity from above for Lebesgue measure, $\vol_d(E_m)\downarrow \vol_d(C_{\z})$.
	Therefore, $\limsup_{k\to\infty} f(\z_k)\le \inf_{m\ge 1}\vol_d(E_m)=\vol_d(C_{\z})=f(\z)$.
	Since $f(\z)=0$ for $\z\notin K$, we also have $\set{\z\in \R^n:\ f(\z)\ge s}\subseteq K$ for every $s>0$.
\end{proof}

We will use the Brunn--Minkowski inequality in the following standard form (see, e.g., \cite{schneider2014convex}).

\begin{theorem}[Brunn--Minkowski]\label{thm:BM}
	Let $A,B\subset \R^d$ be Lebesgue measurable sets with finite volume. For every $\lambda\in [0,1]$,
	\[
		\vol_d\bigl((1-\lambda)A+\lambda B\bigr)^{1/d}\ge (1-\lambda)\vol_d(A)^{1/d}+\lambda \vol_d(B)^{1/d}.
	\]
\end{theorem}

\begin{lemma}[Brunn concavity]\label{lem:brunn-concave}
	Define $g:\R^n\to [0,\infty)$ by
	\begin{equation}\label{eq:g}
		g(\z):=f(\z)^{1/d}.
	\end{equation}
	Then $g$ is concave on $K$ and upper semicontinuous on $\R^n$.
\end{lemma}
\begin{proof}
	Let $\z_1,\z_2\in K$ and $\lambda\in [0,1]$. Put $\z_\lambda:=(1-\lambda)\z_1+\lambda \z_2$.
	By convexity of $C$,
	\[
		(1-\lambda)C_{\z_1}+\lambda C_{\z_2}\subseteq C_{\z_\lambda}.
	\]
		Applying Theorem~\ref{thm:BM} and taking $d$th roots yields
		\[
			\begin{aligned}
				g(\z_\lambda)
				 & =\vol_d(C_{\z_\lambda})^{1/d} \\
				 & \ge (1-\lambda)\vol_d(C_{\z_1})^{1/d}+\lambda \vol_d(C_{\z_2})^{1/d} \\
				 & =(1-\lambda)g(\z_1)+\lambda g(\z_2),
			\end{aligned}
		\]
		which is concavity.

	Upper semicontinuity follows from Lemma~\ref{lem:slice-usc} and continuity/monotonicity of $t\mapsto t^{1/d}$ on $[0,\infty)$.
\end{proof}

\subsection{Thickening and erosion bounds}\label{subsec:thickening}

Let $f,g$ be defined as in~\eqref{eq:f} and~\eqref{eq:g}. For $s\ge 0$, define the superlevel set
\begin{equation}\label{eq:Ks}
	L_s:=\set{\z\in \R^n:\ g(\z)\ge s}.
\end{equation}
If $s>0$, then $L_s\subseteq K$ (Lemma~\ref{lem:slice-usc}).
In particular, $L_s$ is bounded for every $s>0$. Note that $L_0=\R^n$, but its value plays no role in the $s$ integrals below.

\begin{lemma}[$Q$-envelopes]\label{lem:q-envelopes}
	Let $g$ be as defined in~\eqref{eq:g} and $L_s$ be as in \eqref{eq:Ks}, and recall $Q=[-1/2,1/2]^n$.
	Define
	\begin{equation}\label{eq:gpm}
		g_+(\z):=\sup_{\q\in Q}g(\z-\q),\qquad g_-(\z):=\inf_{\q\in Q}g(\z+\q).
	\end{equation}
	Then for every $s\ge 0$,
	\begin{equation}\label{eq:level-identities}
		\set{\z:\ g_+(\z)\ge s}=L_s+Q,\qquad \set{\z:\ g_-(\z)\ge s}=L_s\ominus Q.
	\end{equation}
	In particular, $g_+$ and $g_-$ are upper semicontinuous and hence Lebesgue measurable.
\end{lemma}
\begin{proof}
	Fix $s\ge 0$.
	For $g_+$, we have $\z\in \set{g_+\ge s}$ if and only if there exists $\q\in Q$ such that $g(\z-\q)\ge s$, which is equivalent to $\z-\q\in L_s$ and hence to $\z\in L_s+Q$.
	For $g_-$, we have $\z\in \set{g_-\ge s}$ if and only if $g(\z+\q)\ge s$ for all $\q\in Q$, i.e., $\z+Q\subseteq L_s$, which is precisely $\z\in L_s\ominus Q$.
	This proves \eqref{eq:level-identities}.
	Since $L_s$ is closed and $Q$ is compact, both $L_s+Q$ and $L_s\ominus Q=\bigcap_{\q\in Q}(L_s-\q)$ are closed.
	Hence $g_+$ and $g_-$ are upper semicontinuous and therefore Lebesgue measurable.
\end{proof}
Figure~\ref{fig:lattice-sandwich-3d} provides a complementary geometric view (shown for $n=2$) of the level sets and envelopes $L_s$, $L_s+Q$, and $L_s\ominus Q$ that are integrated in Lemma~\ref{lem:integral-thick}.

\begin{figure}[t]
	\centering
	\tdplotsetmaincoords{65}{115}
	\begin{tikzpicture}[tdplot_main_coords, scale=1.36]
		\def\Rs{2.5}
		\def\Hs{2.5}

		\foreach \xx in {-2,-1,0,1,2} {
			\foreach \yy in {-2,-1,0,1,2} {
				\pgfmathsetmacro{\radsq}{\xx*\xx+\yy*\yy}
				\ifdim \radsq pt < 6.5pt
					\filldraw[fill=figred!15, draw=figred!45, thin]
					(\xx-0.5,\yy-0.5,0) -- (\xx+0.5,\yy-0.5,0) --
					(\xx+0.5,\yy+0.5,0) -- (\xx-0.5,\yy+0.5,0) -- cycle;
					\fill[figred] (\xx,\yy,0) circle (1pt);
					\draw[figred!40, densely dotted, ultra thin] (\xx,\yy,0) -- (\xx,\yy,\Hs);
				\fi
			}
		}

		\draw[figblue, thick] plot[domain=0:360, samples=70, variable=\ang]
		({\Rs*cos(\ang)}, {\Rs*sin(\ang)}, 0);
		\node[figblue, anchor=north west] at (1.9, -1.9, 0) {$L_s$};

		\draw[figgold!80!black, thick, dashed]
		plot[domain=0:90, samples=20, variable=\ang] ({0.5+\Rs*cos(\ang)}, {0.5+\Rs*sin(\ang)}, 0) --
		plot[domain=90:180, samples=20, variable=\ang] ({-0.5+\Rs*cos(\ang)}, {0.5+\Rs*sin(\ang)}, 0) --
		plot[domain=180:270, samples=20, variable=\ang] ({-0.5+\Rs*cos(\ang)}, {-0.5+\Rs*sin(\ang)}, 0) --
		plot[domain=270:360, samples=20, variable=\ang] ({0.5+\Rs*cos(\ang)}, {-0.5+\Rs*sin(\ang)}, 0) -- cycle;
		\node[figgold!80!black] at (3.55, 1.65, 0.18) {$L_s+Q$};

		\def\ths{11.537}
		\def\phis{78.463}
		\draw[figgreen, thick, densely dashed]
		plot[domain={\ths}:{\phis}, samples=20, variable=\ang] ({-0.5+\Rs*cos(\ang)}, {-0.5+\Rs*sin(\ang)}, 0) --
		plot[domain={90+\ths}:{90+\phis}, samples=20, variable=\ang] ({0.5+\Rs*cos(\ang)}, {-0.5+\Rs*sin(\ang)}, 0) --
		plot[domain={180+\ths}:{180+\phis}, samples=20, variable=\ang] ({0.5+\Rs*cos(\ang)}, {0.5+\Rs*sin(\ang)}, 0) --
		plot[domain={270+\ths}:{270+\phis}, samples=20, variable=\ang] ({-0.5+\Rs*cos(\ang)}, {0.5+\Rs*sin(\ang)}, 0) -- cycle;
		\draw[<-, figgreen, thick] (-1.3, 1.5, 0) -- (-2.8, 2.8, 0) node[anchor=south east] {$L_s\ominus Q$};

		\fill[figblue!5, opacity=0.7] (-3.5,-3.5,\Hs) -- (3.5,-3.5,\Hs) -- (3.5,3.5,\Hs) -- (-3.5,3.5,\Hs) -- cycle;
		\draw[figblue!20, thin] (-3.5,-3.5,\Hs) -- (3.5,-3.5,\Hs) -- (3.5,3.5,\Hs) -- (-3.5,3.5,\Hs) -- cycle;
		\node[figblue, anchor=north east] at (-3.4,-3.4,\Hs) {slice level $s$};
		\draw[figblue, thick, fill=figblue!15, opacity=0.8]
		plot[domain=0:360, samples=70, variable=\ang] ({\Rs*cos(\ang)}, {\Rs*sin(\ang)}, \Hs);
		\foreach \xx in {-2,-1,0,1,2} {
			\foreach \yy in {-2,-1,0,1,2} {
				\pgfmathsetmacro{\radsq}{\xx*\xx+\yy*\yy}
				\ifdim \radsq pt < 6.5pt
					\fill[figred] (\xx,\yy,\Hs) circle (1.2pt);
				\fi
			}
		}

		\foreach \zz in {3.0,3.5,4.0,4.5} {
			\pgfmathsetmacro{\rr}{sqrt(max(0,(5.0-\zz)/0.4))}
			\draw[figblue!70, thin, smooth, domain=0:360, samples=40, variable=\ang]
			plot ({\rr*cos(\ang)}, {\rr*sin(\ang)}, \zz);
		}
		\foreach \ang in {0,45,90,135} {
			\draw[figblue!70, thin, smooth, domain=-\Rs:\Rs, samples=30, variable=\xx]
			plot ({\xx*cos(\ang)}, {\xx*sin(\ang)}, {\Hs+2.5-0.4*\xx*\xx});
		}
		\draw[<-, figblue, thick] (0,1.5,4.1) -- (0,2.5,5.0) node[above] {$g(\z)$};
	\end{tikzpicture}
\caption{Geometric view (shown over $\Z^2\times\R$) of the superlevel set $L_s=\set{\z\in\R^2:\ g(\z)\ge s}$ and the sets $L_s+Q$ and $L_s\ominus Q$. At level $s$, the union $\bigcup_{\z\in L_s\cap\Z^2}(\z+Q)$ has $\H_2$ measure $\#(L_s\cap\Z^2)$ and satisfies $L_s\ominus Q\subseteq \bigcup_{\z\in L_s\cap\Z^2}(\z+Q)\subseteq L_s+Q$, which is the geometric content of Lemma~\ref{lem:lattice-sandwich}.}
	\label{fig:lattice-sandwich-3d}
\end{figure}
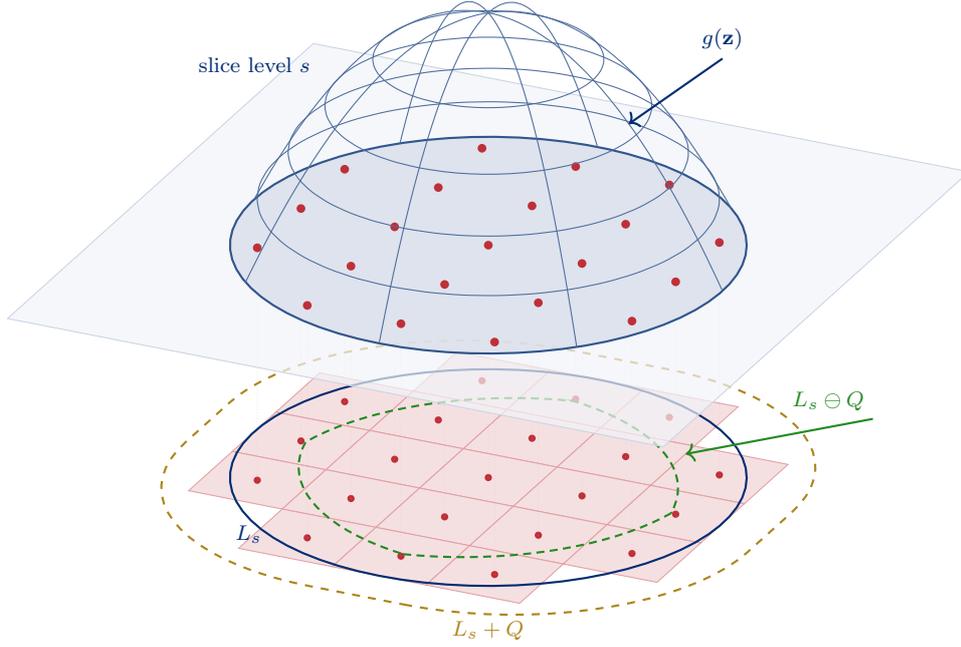

Figure~\ref{fig:morphological-envelope} shows a one dimensional slice of the envelope construction in Lemma~\ref{lem:q-envelopes}, clarifying how $g_+$ and $g_-$ control the discrete layer cake sum via \eqref{eq:level-identities}.

\begin{figure}[t]
	\centering
	\begin{tikzpicture}[x=1.2cm, y=1.35cm]
		\fill[figgold!10]
		plot[domain=-3.2:3.2, samples=120, variable=\xx] (\xx, {max(0,3.0-0.4*max(0,abs(\xx)-0.5)^2)})
		-- (3.2,0) -- (-3.2,0) -- cycle;

		\foreach \xx in {-2.5,-1.5,-0.5,0.5,1.5,2.5} {
			\draw[black!15, densely dotted, thick] (\xx,0) -- (\xx,{max(0,3.0-0.4*max(0,abs(\xx)-0.5)^2)});
		}

		\foreach \zz in {-2,-1,0,1,2} {
			\pgfmathsetmacro{\hh}{max(0,3.0-0.4*\zz*\zz)}
			\filldraw[fill=figred!15, draw=figred!50, thick] (\zz-0.5,0) rectangle (\zz+0.5,\hh);
			\fill[figred] (\zz,\hh) circle (2pt);
		}

		\draw[figblue, semithick, dashed, domain=-2.8:2.8, samples=120, variable=\xx]
		plot (\xx, {max(0,3.0-0.4*\xx*\xx)});
		\draw[figgold!80!black, semithick, domain=-3.2:3.2, samples=120, variable=\xx]
		plot (\xx, {max(0,3.0-0.4*max(0,abs(\xx)-0.5)^2)});
		\draw[figgreen, semithick, domain=-2.3:2.3, samples=120, variable=\xx]
		plot (\xx, {max(0,3.0-0.4*(abs(\xx)+0.5)^2)});

		\draw[->, thick, black!70] (-3.95,0) -- (5.25,0);
		\draw[->, thick, black!70] (0,0) -- (0,4.0) node[above, text=black!90, font=\normalsize] {slice radius};
		\node[anchor=west, text=black!90, font=\small] at (3.2,-0.30) {affine line in $\R^n$};

		\foreach \xx/\lab in {0/\z,1/\z+\u,-1/\z-\u,2/\z+2\u,-2/\z-2\u} {
			\draw[black!70, thick] (\xx,0.05) -- (\xx,-0.05) node[below, font=\normalsize, text=black!80] {$\lab$};
		}

		\draw[figgold!80!black, thin, densely dotted] (-0.5,3.0) -- (-0.5,3.4);
		\draw[figgold!80!black, thin, densely dotted] (0.5,3.0) -- (0.5,3.4);
		\draw[<->, figgold!80!black, thick] (-0.5,3.25) -- (0.5,3.25)
		node[midway, fill=figgold!10, inner sep=2pt, font=\normalsize] {$Q$};

		\node[figblue, align=right, anchor=east] at (-3.2,2.6) {$g(\z)$};
		\draw[->, figblue, thick, shorten >=2pt] (-3.2,2.6) -- (-1.8, {max(0,3.0-0.4*(-1.8)*(-1.8))});
		\node[figgreen, align=right, anchor=east] at (-3.2,1.0) {$g_-(\z)$};
		\draw[->, figgreen, thin, shorten >=2pt] (-3.2,1.0) -- (-1.5, 1.4);
		\node[figgold!80!black, align=left, anchor=west] at (3.2,2.8) {$g_+(\z)$};
		\draw[->, figgold!80!black, thin, shorten >=2pt] (3.2,2.8) -- (1.5, 2.6);
	\end{tikzpicture}
	\caption{One dimensional restriction along an affine line in $\R^n$ of the envelope identities in Lemma~\ref{lem:q-envelopes}. The blue dashed curve is $g$, and the gold and green curves are $g_+$ and $g_-$. The red columns correspond to unit cells centered at integer points on the slice. Equation \eqref{eq:level-identities} identifies the superlevel sets of $g_+$ and $g_-$ with $L_s+Q$ and $L_s\ominus Q$, which are integrated in \eqref{eq:int-thick} and \eqref{eq:int-erode}.}
	\label{fig:morphological-envelope}
\end{figure}
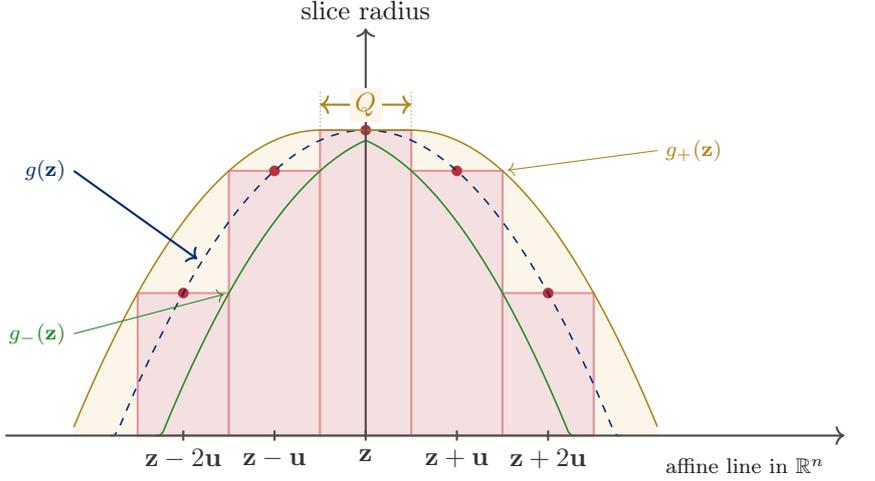

\begin{lemma}[Integral thickening and erosion]\label{lem:integral-thick}
	Let $K\subset \R^n$ be a convex body and let $g:\R^n\to [0,\infty)$ be as defined in~\eqref{eq:g} and let $L_s$ be as in \eqref{eq:Ks}.
	Assume that there exist $\c\in \R^n$ and $k>0$ such that
	\begin{equation}\label{eq:2kQ}
		\c+2kQ\subseteq K.
	\end{equation}
	Then for every integer $d\ge 1$,
	\begin{equation}\label{eq:int-thick}
		d\int_0^\infty s^{d-1}\vol_n(L_s+Q)\dd s \le \Bigl(1+\frac{1}{2k}\Bigr)^{n+d}
		d\int_0^\infty s^{d-1}\vol_n(L_s)\dd s.
	\end{equation}
	If in addition $k>1/2$, then
	\begin{equation}\label{eq:int-erode}
		d\int_0^\infty s^{d-1}\vol_n(L_s\ominus Q)\dd s \ge \Bigl(1-\frac{1}{2k}\Bigr)^{n+d}
		d\int_0^\infty s^{d-1}\vol_n(L_s)\dd s.
	\end{equation}
\end{lemma}
\begin{proof}
	Translate coordinates so that $\c=\0$. This does not change the integrals in \eqref{eq:int-thick} and \eqref{eq:int-erode}.
	Thus we assume $2kQ\subseteq K$.

	Define the set
	\[
		E:=\set{(\z,\x)\in \R^n\times \R^d:\ \z\in K,\ \|\x\|_2\le g(\z)}.
	\]
	Each $\z$ fiber is replaced by a centered Euclidean ball of radius $g(\z)$, a fiberwise Schwarz symmetrization \citep{kesavan2006symmetrization}.
	The set $E$ is convex: if $(\z_i,\x_i)\in E$ and $\lambda\in [0,1]$, then
	$\z_\lambda=(1-\lambda)\z_1+\lambda \z_2\in K$ by convexity of $K$, and
	\[
		\|\,(1-\lambda)\x_1+\lambda \x_2\,\|_2\le (1-\lambda)\|\x_1\|_2+\lambda \|\x_2\|_2\le (1-\lambda)g(\z_1)+\lambda g(\z_2)\le g(\z_\lambda),
	\]
	where the last inequality uses concavity of $g$ on $K$ (Lemma~\ref{lem:brunn-concave}). Hence $(\z_\lambda,(1-\lambda)\x_1+\lambda \x_2)\in E$.

	Let $v_d:=\vol_d(\set{\x:\ \|\x\|_2\le 1})$, the $d$-dimensional Lebesgue measure of the Euclidean unit ball in $\R^d$.
	For each $\z\in K$, the $\x$ fiber $\set{\x:\ \|\x\|_2\le g(\z)}$ equals the dilation $g(\z)\set{\x:\ \|\x\|_2\le 1}$, so
	\[\vol_d(\set{\x:\ \|\x\|_2\le g(\z)})=v_d\,g(\z)^d.\]
	Since $g(\z)=0$ on $\R^n\setminus K$, we have $\int_K g(\z)^d\dd \z=\int_{\R^n}g(\z)^d\dd \z$.
	For each $\z\in \R^n$, let $E_\z:=\set{\x\in \R^d:\ (\z,\x)\in E}$.
	Then $E_\z=\set{\x:\ \|\x\|_2\le g(\z)}$ if $\z\in K$, and $E_\z=\emptyset$ otherwise.
	In particular, $\int_{\R^d}\1_E(\z,\x)\dd\x=\vol_d(E_\z)$ for every $\z\in \R^n$.
	By Tonelli's theorem, we have
	\[
		\begin{aligned}
			\vol_{n+d}(E)
			 & =\int_{\R^{n+d}}\1_E(\z,\x)\dd(\z,\x)
			=\int_{\R^n}\int_{\R^d}\1_E(\z,\x)\dd\x\,\dd\z \\
			 & =\int_{\R^n}\vol_d(E_\z)\dd\z
			=\int_{\z\in K}\vol_d(\set{\x:\ \|\x\|_2\le g(\z)})\dd\z.
		\end{aligned}
	\]
	\noindent
	Combining with the identities above yields
	\begin{equation}\label{eq:volD}
		\vol_{n+d}(E)=v_d\int_{\R^n}g(\z)^d\dd \z.
	\end{equation}
	Since $g$ is upper semicontinuous by Lemma~\ref{lem:brunn-concave}, it is Lebesgue measurable. Applying Lemma~\ref{lem:layer-cake} to $g^d$ and using $\set{g\ge s}=L_s$ gives
	\begin{equation}\label{eq:intKs}
		\int_{\R^n}g(\z)^d\dd \z = d\int_0^\infty s^{d-1}\vol_n(L_s)\dd s.
	\end{equation}

		Next consider $E^+:=E+\left(Q\times \{\0\}\right)$.
		For each $\z\in \R^n$, write $E^+_{\z}$ for the $\z$-fiber of $E^+$ (i.e., the $\x$-section at $\z$).
		Then $\x\in E^+_{\z}$ if and only if $(\z,\x)\in E^+$, i.e., there exists $\q\in Q$ such that $(\z-\q,\x)\in E$. This is equivalent to saying there exists $\q\in Q$ with $\z-\q\in K$ and $\|\x\|_2\le g(\z-\q)$.
		Thus,
		\[
			E^+_{\z}
			=\bigcup_{\substack{\q\in Q\\ \z-\q\in K}} g(\z-\q)\,\mathbb{B}_2^d.
		\]
	If the index set is empty, then $E^+_{\z}$ is empty and has Lebesgue measure $0$.
	Otherwise, the index set is compact and the map $\q\mapsto g(\z-\q)$ is upper semicontinuous, so it attains its maximum at some $\q_\z\in Q$ with $\z-\q_\z\in K$.
	The ball $\set{\x:\ \|\x\|_2\le g(\z-\q_\z)}$ contains all other balls in the union, and hence the union equals this ball.
	Its radius satisfies $g(\z-\q_\z)=\max_{\q\in Q}g(\z-\q)=g_+(\z)$, so the fiber $E^+_{\z}$ has $d$-dimensional Lebesgue measure $v_d\,g_+(\z)^d$.
	Applying Tonelli's theorem yields
	\[
		\vol_{n+d}(E^+)
		=\int_{\R^n}\vol_d(E^+_{\z})\dd\z
		=v_d\int_{\R^n} g_+(\z)^d\dd \z.
	\]
	By Lemma~\ref{lem:q-envelopes}, $\set{\z:\ g_+(\z)\ge s}=L_s+Q$ for every $s\ge 0$, and $g_+$ is Lebesgue measurable.
	Applying Lemma~\ref{lem:layer-cake} to $g_+^d$ yields
	\begin{equation}\label{eq:volDthick}
		\vol_{n+d}(E^+) = v_d\,
		d\int_0^\infty s^{d-1}\vol_n(L_s+Q)\dd s.
	\end{equation}

		Now consider the erosion $E^-:=E\ominus\left(Q\times\{\0\}\right)$.
		For each $\z\in \R^n$, write $E^-_{\z}$ for the $\z$-fiber of $E^-$ (i.e., the $\x$-section at $\z$).
		By definition of Minkowski erosion, $\x\in E^-_{\z}$ if and only if $(\z,\x)+(\q,\0)\in E$ for every $\q\in Q$, i.e., $\z+Q\subseteq K$ and $\|\x\|_2\le g(\z+\q)$ for all $\q\in Q$.
		Hence $E^-_{\z}=\emptyset$ unless $\z+Q\subseteq K$, in which case
		\[
			\begin{aligned}
				E^-_{\z}
				 & =\bigcap_{\q\in Q}\set{\x:\ \|\x\|_2\le g(\z+\q)}\\
				& =\set{\x:\ \|\x\|_2\le \inf_{\q\in Q} g(\z+\q)} \\
				 & =\set{\x:\ \|\x\|_2\le g_-(\z)}.
			\end{aligned}
		\]
	Since all balls are centered at the origin, the intersection is again a ball, whose radius is the infimum of the radii.
	If $\z+Q\not\subseteq K$, then $E^-_{\z}=\emptyset$ and $g_-(\z)=0$ because there exists $\q\in Q$ with $\z+\q\notin K$ and hence $g(\z+\q)=0$.
	If $\z+Q\subseteq K$, then $\inf_{\q\in Q} g(\z+\q)=\inf_{\q\in Q}g(\z+\q)=g_-(\z)$.
	In either case, $\vol_d(E^-_{\z})=v_d\,g_-(\z)^d$.
	Applying Tonelli's theorem yields
	\[
		\vol_{n+d}\Bigl(E\ominus\left(Q\times\{\0\}\right)\Bigr)
		=\vol_{n+d}(E^-)
		=\int_{\R^n}\vol_d(E^-_{\z})\dd\z
		=v_d\int_{\R^n} g_-(\z)^d\dd \z.
	\]
	By Lemma~\ref{lem:q-envelopes}, $\set{\z:\ g_-(\z)\ge s}=L_s\ominus Q$ for every $s\ge 0$, and $g_-$ is Lebesgue measurable.
	Applying Lemma~\ref{lem:layer-cake} to $g_-^d$ yields
	\begin{equation}\label{eq:volDerode}
		\vol_{n+d}(E\ominus\left(Q\times\{\0\}\right)) = v_d\,
		d\int_0^\infty s^{d-1}\vol_n(L_s\ominus Q)\dd s.
	\end{equation}

	We now compare these volumes to $\vol_{n+d}(E)$ by a scaling argument. Let $\epsilon:=1/(2k)$. The key observation is that the convexity of $E$ allows us to control its Minkowski sum by a simple dilation: since $E$ contains a translate of $2kQ\times\{\0\}$, adding $Q\times\{\0\}$ to $E$ enlarges it by at most a factor of $(1+\epsilon)$.
	Since $2kQ\subseteq K$, for every $\q\in Q$ the point $(2k\q,\0)$ lies in $E$.
	Fix $(\z,\x)\in E$ and $\q\in Q$. Then
	\[
		(\z,\x)+(\q,\0) = (\z,\x)+\epsilon(2k\q,\0) = (1+\epsilon)\Bigl(\frac{1}{1+\epsilon}(\z,\x)+\frac{\epsilon}{1+\epsilon}(2k\q,\0)\Bigr).
	\]
	The convex combination in parentheses belongs to $E$ by convexity of $E$. Hence
	\begin{equation}\label{eq:contain-thick}
		E^+\subseteq (1+\epsilon)E.
	\end{equation}
	Taking volumes gives
	\[
		\vol_{n+d}(E^+)\le (1+\epsilon)^{n+d}\vol_{n+d}(E).
	\]
	Substituting \eqref{eq:volD}, \eqref{eq:intKs}, and \eqref{eq:volDthick} into this inequality and recalling that $\epsilon=1/(2k)$ gives
	\[
		v_d\,d\int_0^\infty s^{d-1}\vol_n(L_s+Q)\dd s \le \Bigl(1+\frac{1}{2k}\Bigr)^{n+d} v_d\,d\int_0^\infty s^{d-1}\vol_n(L_s)\dd s.
	\]
	Canceling $v_d$ yields \eqref{eq:int-thick}.

	For erosion, assume $k>1/2$, so $\epsilon\in (0,1)$. Fix $(\z,\x)\in E$ and $\q\in Q$. Convexity of $E$ gives
	\[
		(1-\epsilon)(\z,\x)+\epsilon(2k\q,\0)\in E.
	\]
	Since $\epsilon(2k\q,\0)=(\q,\0)$, this means
	\[
		(1-\epsilon)(\z,\x)+(\q,\0)\in E\qquad \forall \q\in Q,
	\]
	which is equivalent to $(1-\epsilon)(\z,\x)\in E\ominus\left(Q\times\{\0\}\right)$. Therefore,
	\begin{equation}\label{eq:contain-erode}
		(1-\epsilon)E\subseteq E\ominus\left(Q\times\{\0\}\right).
	\end{equation}
	Taking volumes gives
	\[
		(1-\epsilon)^{n+d}\vol_{n+d}(E)\le \vol_{n+d}\Bigl(E\ominus\left(Q\times\{\0\}\right)\Bigr).
	\]
	Substituting \eqref{eq:volD}, \eqref{eq:intKs}, and \eqref{eq:volDerode} into this inequality and recalling that $\epsilon=1/(2k)$ gives
	\[
		\Bigl(1-\frac{1}{2k}\Bigr)^{n+d} v_d\,d\int_0^\infty s^{d-1}\vol_n(L_s)\dd s \le v_d\,d\int_0^\infty s^{d-1}\vol_n(L_s\ominus Q)\dd s.
	\]
	Canceling $v_d$ yields \eqref{eq:int-erode}.
\end{proof}

\begin{figure}[t]
	\centering
	\begin{tikzpicture}[x=1.38cm, y=1.18cm]

		\fill[figgold!27]
			plot[domain=-3:3, samples=120, variable=\tt]
			(\tt, {max(0, 3.0-0.48*max(0,abs(\tt)-0.5)*max(0,abs(\tt)-0.5))})
			-- plot[domain=3:-3, samples=120, variable=\tt]
			(\tt, {-max(0, 3.0-0.48*max(0,abs(\tt)-0.5)*max(0,abs(\tt)-0.5))}) -- cycle;
		\fill[figblue!18]
			plot[domain=-2.5:2.5, samples=120, variable=\tt]
			(\tt, {max(0, 3.0-0.48*\tt*\tt)})
			-- plot[domain=2.5:-2.5, samples=120, variable=\tt]
			(\tt, {-max(0, 3.0-0.48*\tt*\tt)}) -- cycle;
		\fill[figgreen!22]
			plot[domain=-2:2, samples=120, variable=\tt]
			(\tt, {max(0, 3.0-0.48*(abs(\tt)+0.5)*(abs(\tt)+0.5))})
			-- plot[domain=2:-2, samples=120, variable=\tt]
			(\tt, {-max(0, 3.0-0.48*(abs(\tt)+0.5)*(abs(\tt)+0.5))}) -- cycle;

		\draw[figblue!75!black, thick, dashed, domain=-3:3, samples=120, variable=\tt]
			plot (\tt, {max(0, 3.6-0.4*\tt*\tt)});
		\draw[figblue!75!black, thick, dashed, domain=-3:3, samples=120, variable=\tt]
			plot (\tt, {-max(0, 3.6-0.4*\tt*\tt)});
		\draw[figgold!80!black, semithick, domain=-3:3, samples=120, variable=\tt]
			plot (\tt, {max(0, 3.0-0.48*max(0,abs(\tt)-0.5)*max(0,abs(\tt)-0.5))});
		\draw[figgold!80!black, semithick, domain=-3:3, samples=120, variable=\tt]
			plot (\tt, {-max(0, 3.0-0.48*max(0,abs(\tt)-0.5)*max(0,abs(\tt)-0.5))});
		\draw[figblue, thick, domain=-2.5:2.5, samples=120, variable=\tt]
			plot (\tt, {max(0, 3.0-0.48*\tt*\tt)});
		\draw[figblue, thick, domain=-2.5:2.5, samples=120, variable=\tt]
			plot (\tt, {-max(0, 3.0-0.48*\tt*\tt)});
		\draw[figgreen, semithick, domain=-2:2, samples=120, variable=\tt]
			plot (\tt, {max(0, 3.0-0.48*(abs(\tt)+0.5)*(abs(\tt)+0.5))});
		\draw[figgreen, semithick, domain=-2:2, samples=120, variable=\tt]
			plot (\tt, {-max(0, 3.0-0.48*(abs(\tt)+0.5)*(abs(\tt)+0.5))});
		\draw[figblue!45!black, thick, dashed, domain=-2:2, samples=120, variable=\tt]
			plot (\tt, {max(0, 2.4-0.6*\tt*\tt)});
		\draw[figblue!45!black, thick, dashed, domain=-2:2, samples=120, variable=\tt]
			plot (\tt, {-max(0, 2.4-0.6*\tt*\tt)});

		\draw[->, thick, black!70] (-3.6,0) -- (4.2,0) node[right, font=\small, black!90] {$\z$};
		\draw[->, thick, black!70] (0,-3.9) -- (0,4.7) node[above, font=\small, black!90] {$\x$};
		\draw[black!45, thin] (2.5, 0.09) -- (2.5,-0.09) node[below, font=\scriptsize, black!60] {$k$};
		\draw[black!45, thin] (-2.5, 0.09) -- (-2.5,-0.09) node[below, font=\scriptsize, black!60] {$-k$};

		\node[figblue, anchor=south, font=\small] at (0, 3.08) {$g(\z)$};
		\node[figblue, font=\small, anchor=west] at (3.2, 0.5) {$E$};
		\draw[->, figblue, thin, shorten >=3pt]
			(3.18, 0.5) -- (2.1, 0.38);
		\node[figgold!80!black, font=\small, anchor=west, align=left] at (3.2, 1.9)
			{$E^+$\\$g_+(\z)$};
		\draw[->, figgold!80!black, thin, shorten >=3pt]
			(3.18, 1.9) -- (2.0, {3.0-0.48*(2.0-0.5)*(2.0-0.5)});
		\node[figgreen, font=\small, anchor=east, align=right] at (-3.2, 1.9)
			{$E^-$\\$g_-(\z)$};
		\draw[->, figgreen, thin, shorten >=3pt]
			(-3.18, 1.9) -- (-1.3, {3.0-0.48*(1.3+0.5)*(1.3+0.5)+0.05});
		\node[figblue!75!black, anchor=west, font=\scriptsize] at (3.2, 3.8) {$(1+\epsilon)E$};
		\draw[->, figblue!75!black, thin, densely dashed, shorten >=3pt]
			(3.18, 3.8) -- (1.8, {3.6-0.4*1.8*1.8+0.1});
		\node[figblue!45!black, anchor=east, font=\scriptsize] at (-3.2, 3.8) {$(1-\epsilon)E$};
		\draw[->, figblue!45!black, thin, densely dashed, shorten >=3pt]
			(-3.18, 3.8) -- (-0.8, {2.4-0.6*0.8*0.8-0.05});

		\node[black!65, font=\scriptsize, anchor=north] at (0, -4.3)
			{$(1-\epsilon)E \;\subseteq\; E^- \;\subseteq\; E \;\subseteq\; E^+ \;\subseteq\; (1+\epsilon)E,
			\qquad \epsilon = \tfrac{1}{2k}$};

		\draw[black!45, thin] (3.0, 0.07) -- (3.0,-0.07);
		\draw[figgold!80!black, thin, densely dotted] (2.5,-0.22) -- (2.5,-0.52);
		\draw[figgold!80!black, thin, densely dotted] (3.0,-0.07) -- (3.0,-0.52);
		\draw[<->, figgold!80!black, thin] (2.5,-0.45) -- (3.0,-0.45)
			node[midway, below, font=\scriptsize] {$Q$};
	\end{tikzpicture}
	\caption{Geometric view (shown for $n=d=1$) of the symmetrized body
		$E=\set{(\z,\x)\in \R^n\times\R^d:\ \z\in K,\ \|\x\|_2\le g(\z)}$ (blue) together with its Minkowski
		thickening $E^+=E+(Q\times\{\0\})$ (gold) and erosion
		$E^-=E\ominus(Q\times\{\0\})$ (green).
		For $\z\in \R^n$, the $\x$ fibers of $E$, $E^+$, and $E^-$ are centered Euclidean balls whenever they are nonempty, with radii given by $g(\z)$, $g_+(\z)$, and $g_-(\z)$ from \eqref{eq:gpm}.
		The dashed sets depict the dilations $(1\pm\epsilon)E$ with $\epsilon=1/(2k)$, and \eqref{eq:contain-thick} and \eqref{eq:contain-erode} compare these dilations with $E^+$ and $E^-$.}
	\label{fig:E-bodies}
\end{figure}

\subsection{Proof of the comparison inequality}\label{subsec:comparison-proof}

\begin{proof}[of Lemma~\ref{lem:disc-cont}]
	For $\z\in \R^n$, define
	\[
		f(\z):=\vol_d(D_{\z}),\qquad g(\z):=f(\z)^{1/d}.
	\]
	Applying Lemma~\ref{lem:slice-usc} to $D$ in place of $C$, the function $f$ is upper semicontinuous on $\R^n$.
	Applying Lemma~\ref{lem:brunn-concave} to $D$ in place of $C$, the function $g$ is concave on $K_D$ and upper semicontinuous on $\R^n$.
	For $s\ge 0$, define $L_s=\set{\z:\ g(\z)\ge s}$.
	In particular, $g$ is Lebesgue measurable.

	By Fubini's theorem,
	\[
		\vol_{n+d}(D)=\int_{\R^n}\vol_d(D_{\z})\dd \z=\int_{K_D} f(\z)\dd \z.
	\]
	Since $g^d=f$ on $\R^n$, we also have $\int_{K_D} f=\int_{\R^n}g^d$.
	Applying Lemma~\ref{lem:layer-cake} to $g^d$ gives
	\begin{equation}\label{eq:volD-layer}
		\vol_{n+d}(D) = \int_{K_D} f(\z)\dd \z = \int_{\R^n} g(\z)^d\dd \z = d\int_0^\infty s^{d-1}\vol_n(L_s)\dd s.
	\end{equation}

	On the other hand, since $D$ is bounded, only finitely many $\z\in \Z^n$ yield a nonempty fiber $D\cap\left(\{\z\}\times \R^d\right)$, and thus
	\begin{equation}\label{eq:nu-sum}
		\H_d(S_D) = \sum_{\z\in \Z^n}\vol_d\bigl(D\cap\left(\{\z\}\times \R^d\right)\bigr) = \sum_{\z\in \Z^n} f(\z) = \sum_{\z\in \Z^n} g(\z)^d.
	\end{equation}
	For each $\z\in \Z^n$ we use the elementary identity
	\[
		g(\z)^d = d\int_0^{g(\z)} s^{d-1}\dd s
		= d\int_0^\infty s^{d-1}\1_{\{g(\z)\ge s\}}\dd s,
	\]
	where the second equality rewrites the integral over $[0,g(\z)]$ using an indicator function.
	Since the integrand is nonnegative, Tonelli's theorem allows us to exchange the sum and the integral in \eqref{eq:nu-sum}.
	Moreover, for each $s>0$ we have
		$\sum_{\z\in \Z^n}\1_{\{g(\z)\ge s\}}=\#\set{\z\in \Z^n:\ g(\z)\ge s}=\#(L_s\cap \Z^n)$.
	This quantity is finite since $L_s\subseteq K_D$ for $s>0$ and $K_D$ is compact.
	Moreover, the value at $s=0$ does not affect the integral.
	This yields
		\begin{equation}\label{eq:nu-layer}
			\H_d(S_D) = d \int_0^\infty s^{d-1} \sum_{\z\in \Z^n} \1_{\{g(\z)\ge s\}} \dd s = d\int_0^\infty s^{d-1}\#(L_s\cap \Z^n)\dd s.
		\end{equation}

		For each $s>0$, the set $L_s$ is closed (Lemma~\ref{lem:brunn-concave}), so Lemma~\ref{lem:lattice-sandwich} applies and yields
		\[
			\vol_n(L_s\ominus Q)\le \#(L_s\cap \Z^n)\le \vol_n(L_s+Q).
		\]
	Multiplying by $d s^{d-1}$ and integrating in $s$ gives
		\begin{equation}\label{eq:nu-sandwich-int}
			d\int_0^\infty s^{d-1}\vol_n(L_s\ominus Q)\dd s \ \le\ \H_d(S_D) \ \le\ d\int_0^\infty s^{d-1}\vol_n(L_s+Q)\dd s.
		\end{equation}

	Assumption \eqref{eq:k-box} is equivalent to $\z_D+2kQ\subseteq K_D$.
	Thus Lemma~\ref{lem:integral-thick} applies to $K_D$ with center $\c=\z_D$ and parameter $k$.
	Combining \eqref{eq:nu-sandwich-int} with \eqref{eq:int-thick} and \eqref{eq:volD-layer} gives
		\[
			\H_d(S_D)\le \Bigl(1+\frac{1}{2k}\Bigr)^{n+d}d\int_0^\infty s^{d-1}\vol_n(L_s)\dd s = \Bigl(1+\frac{1}{2k}\Bigr)^{n+d}\vol_{n+d}(D).
		\]
	Similarly, combining \eqref{eq:nu-sandwich-int} with \eqref{eq:int-erode} and \eqref{eq:volD-layer} gives
		\[
			\H_d(S_D)\ge \Bigl(1-\frac{1}{2k}\Bigr)^{n+d}d\int_0^\infty s^{d-1}\vol_n(L_s)\dd s = \Bigl(1-\frac{1}{2k}\Bigr)^{n+d}\vol_{n+d}(D).
		\]
	This proves \eqref{eq:disc-cont-bound}.
\end{proof}

\section{Proof of main results}\label{sec:proof}

We now prove Theorem~\ref{thm:main} and Corollary~\ref{cor:main}.
Recall that $S:=C\cap(\Z^n\times\R^d)$. We build a mixed-integer point $\y^\star\in S$ by shrinking $C$ to a body $C'\subseteq C$ whose centroid has an integer $\R^n$-projection (Lemma~\ref{lem:centroid-rounding}), and we set $\y^\star:=\mathfrak{c}(C') \in S$. Fixing a halfspace $H$ with $\y^\star\in H$, Grünbaum's inequality (Theorem~\ref{thm:grunbaum}) controls the Lebesgue volumes of the two pieces of $C'$ cut by $\partial H$. A geometric bisection argument (Lemma~\ref{lem:half-bisect}) guarantees that one of the two pieces contains a translate of $\tfrac12 C'$, hence still has a sufficiently large projection. We then invoke the comparison lemma (Lemma~\ref{lem:disc-cont}) to convert these Lebesgue-volume bounds into mixed-integer volume bounds for $S$. A short case analysis handles whether the ``thick'' piece coincides with $C'\cap H$ or with its complement, and a Bernoulli-type estimate yields \eqref{eq:main-bound}.

\subsection{Auxiliary lemmas}

The first lemma allows us to restrict attention to halfspaces whose boundary passes through the candidate centerpoint.

\begin{lemma}[Boundary reduction for depth]\label{lem:boundary-reduction}
	For every $\y\in S$,
	\[
		h_S(\y)=\inf\set{\H_d(S\cap H):\ H \text{ is a closed halfspace and } \y\in \partial H}.
	\]
\end{lemma}
\begin{proof}
	Fix $\y\in S$, and let $H$ be a closed halfspace with $\y\in H$.
	Write $H=\set{\u\in \R^{n+d}:\ \a^\T\u\ge c}$ with $\a\neq \0$.
	If $\a^\T\y=c$, then $\y\in\partial H$ and there is nothing to show.
	Otherwise, $\a^\T\y>c$.
	Set $c_0:=\a^\T\y$ and let
	\[
		H_0:=\set{\u\in \R^{n+d}:\ \a^\T\u\ge c_0}.
	\]
	Then $\y\in \partial H_0$ and $H_0\subseteq H$.
	By monotonicity of $\H_d$, $\H_d(S\cap H_0)\le \H_d(S\cap H)$.
	Taking the infimum over all $H$ gives the claim.
\end{proof}

A natural candidate for the centerpoint $\y^\star$ is the centroid $\mathfrak{c}(C)$ of the convex body $C$, since Grünbaum's inequality guarantees a halfspace depth of at least $1/e$ for the $(n+d)$ dimensional Lebesgue measure. However, the centroid may not belong to the mixed-integer lattice $\Lambda=\Z^n\times \R^d$ because its $\z$ component is typically not an integer vector. The following lemma addresses this by shrinking $C$ toward a suitably chosen point $\w^\star$, producing a smaller body $C'$ whose centroid has an integer $\z$ component while retaining most of the original volume.

\begin{lemma}[Centroid rounding]\label{lem:centroid-rounding}
	Let $C\subset \R^{n+d}$ be a convex body and set $K=\proj_{\R^n}(C)$.
	Assume that there exist $\z_0\in \R^n$ and $k>1/2$ such that $\z_0+k\mathbb{B}_\infty^n\subseteq K$.
	Let $\epsilon:=1/(2k)$.
		Then there exists $\w^\star\in C$ such that the set $C':=(1-\epsilon)C+\epsilon \w^\star$ satisfies:
	\begin{enumerate}
		\item $C'\subseteq C$ and $\vol_{n+d}(C')=(1-\epsilon)^{n+d}\vol_{n+d}(C)$.
		\item $\proj_{\R^n}(\mathfrak{c}(C'))\in \Z^n$.
		\item $\proj_{\R^n}(C')$ contains a translate of $(k-\tfrac12)\mathbb{B}_\infty^n$.
	\end{enumerate}
\end{lemma}
\begin{proof}
		For any $\w\in C$, define
		\[
			T_{\w}(\y):=(1-\epsilon)\y+\epsilon \w.
		\]
		For any $\y\in C$, the point $T_{\w}(\y)$ is a convex combination of $\y$ and $\w$. Since $C$ is convex and $\y,\w\in C$, we have $T_{\w}(\y)\in C$.
		Therefore, $T_{\w}(C)\subseteq C$.

		Let $DT_{\w}$ denote the Jacobian matrix of $T_{\w}$. Since $T_{\w}$ is affine, $DT_{\w}$ is constant and equals its linear part. Thus $DT_{\w}=(1-\epsilon)I$, so $|\det DT_{\w}|=(1-\epsilon)^{n+d}$. Hence
		\[
			\vol_{n+d}(T_{\w}(C))=(1-\epsilon)^{n+d}\vol_{n+d}(C),
		\]
		proving the first item for $T_{\w}(C)$.

	Write $\mathfrak{c}(C)=(\z_C,\x_C)$. By the change of variables formula, centroids are affine equivariant, i.e.,
		$\mathfrak{c}(T(C))=T(\mathfrak{c}(C))$ for any affine $T$. In particular,
		\[
			\mathfrak{c}(T_{\w}(C))=(1-\epsilon)\mathfrak{c}(C)+\epsilon \w,\qquad \proj_{\R^n}(\mathfrak{c}(T_{\w}(C)))=(1-\epsilon)\z_C+\epsilon \proj_{\R^n}(\w).
		\]
		We now choose $\w^\star$ so that the right hand side in the second equality is an integer vector.

	The assumption $\z_0+k\mathbb{B}_\infty^n\subseteq K$ is equivalent to $\z_0+2kQ\subseteq K$.
	Multiplying by $\epsilon=1/(2k)$ gives $\epsilon \z_0+Q\subseteq \epsilon K$, and therefore
	\[
		(1-\epsilon)\z_C+\epsilon \z_0+Q\subseteq (1-\epsilon)\z_C+\epsilon K.
	\]
	The left hand side is a translate of $Q$, so it contains an integer vector $\z^\star\in \Z^n$.
	By the inclusion above, $\z^\star\in (1-\epsilon)\z_C+\epsilon K$, and hence there exists $\u\in K$ such that
		\[
			(1-\epsilon)\z_C+\epsilon \u=\z^\star.
		\]
		Choose $\w^\star\in C$ with $\proj_{\R^n}(\w^\star)=\u$ (such a $\w^\star$ exists since $\u\in \proj_{\R^n}(C)=K$).
		Now fix this choice and set $C':=T_{\w^\star}(C)$.
		Then
		\[
			\proj_{\R^n}(\mathfrak{c}(C'))=\proj_{\R^n}(\mathfrak{c}(T_{\w^\star}(C)))=\z^\star\in \Z^n,
		\]
		proving the second item.

	Finally, since $\proj_{\R^n}$ is linear, we have
	\[
			\proj_{\R^n}(C')=\proj_{\R^n}(T_{\w^\star}(C))=(1-\epsilon)\proj_{\R^n}(C)+\epsilon \proj_{\R^n}(\w^\star)=(1-\epsilon)K+\epsilon \u.
		\]
	If $K$ contains a translate of $k\mathbb{B}_\infty^n$, then $\proj_{\R^n}(C')$ contains a translate of $(1-\epsilon)k\mathbb{B}_\infty^n$.
	Since $(1-\epsilon)k=k-\tfrac12$, this proves the third item.
\end{proof}

Figure~\ref{fig:centroid-rounding} illustrates the centroid rounding construction of Lemma~\ref{lem:centroid-rounding} for the case $n=d=1$.

\begin{figure}[t]
	\centering
	\begin{tikzpicture}[x=1.39cm, y=1.39cm]
		\fill[figblue!10]
			plot[smooth cycle, tension=0.85] coordinates {
				(-1.8,-0.3) (-1.2,1.6) (0.2,2.4) (1.8,2.1) (3.0,0.6) (2.6,-1.0) (1.2,-1.8) (-0.6,-1.6)
			};

		\fill[figgold!22]
			plot[smooth cycle, tension=0.85] coordinates {
				(-0.96,-0.08) (-0.48,1.44) (0.64,2.08) (1.92,1.84) (2.88,0.64) (2.56,-0.64) (1.44,-1.28) (0.00,-1.12)
			};

		\foreach \zz in {-2,-1,0,1,2,3} {
			\draw[figred!35, line width=0.8pt] (\zz, -2.6) -- (\zz, 3.0);
		}

		\draw[figblue, thick]
			plot[smooth cycle, tension=0.85] coordinates {
				(-1.8,-0.3) (-1.2,1.6) (0.2,2.4) (1.8,2.1) (3.0,0.6) (2.6,-1.0) (1.2,-1.8) (-0.6,-1.6)
			};
		\draw[figgold!80!black, thick]
			plot[smooth cycle, tension=0.85] coordinates {
				(-0.96,-0.08) (-0.48,1.44) (0.64,2.08) (1.92,1.84) (2.88,0.64) (2.56,-0.64) (1.44,-1.28) (0.00,-1.12)
			};

		\draw[->, thick, black!70] (-2.4,0) -- (3.8,0) node[right, font=\small, text=black!90] {$\z$};
		\draw[->, thick, black!70] (0,-2.4) -- (0,3.1) node[above, font=\small, text=black!90] {$\x$};

		\foreach \zz in {-2,-1,1,2,3} {
			\draw[black!60, thick] (\zz,0.06) -- (\zz,-0.06);
			\fill[figred!55] (\zz,0) circle (1.6pt);
		}

		\draw[figred!50, semithick, densely dashed] (0.65, 0.25) -- (2.40, 0.80);

		\draw[black!55, thin, decorate, decoration={brace, amplitude=3pt, mirror}]
			([yshift=-3pt] 0.65, 0.25) -- ([yshift=-3pt] 1.00, 0.36)
			node[midway, below=3.5pt, font=\scriptsize, black!65, sloped] {$\epsilon$};
		\draw[black!55, thin, decorate, decoration={brace, amplitude=3pt, mirror}]
			([yshift=-3pt] 1.00, 0.36) -- ([yshift=-3pt] 2.40, 0.80)
			node[midway, below=3.5pt, font=\scriptsize, black!65, sloped] {$1{-}\epsilon$};

		\fill[figblue] (0.65, 0.25) circle (2.2pt);
		\draw[figblue, densely dotted, thin] (0.65, 0.25) -- (0.65, 0);
		\fill[figblue] (0.65, 0) circle (1.2pt);
		\node[figblue, font=\small, anchor=south east] at (0.58, 0.28) {$\mathfrak{c}(C)$};

		\fill[figgold!80!black] (1.00, 0.36) circle (2.2pt);
		\draw[figgold!80!black, densely dotted, thin] (1.00, 0.36) -- (1.00, 0);
		\fill[figgold!80!black] (1.00, 0) circle (1.6pt);
		\node[figgold!80!black, font=\small, anchor=south west] at (1.08, 0.40) {$\mathfrak{c}(C')$};

		\fill[figred] (2.40, 0.80) circle (2.2pt);
		\node[figred, font=\small, anchor=south west] at (2.25, 0.85) {$\w^\star$};

		\node[figblue, font=\scriptsize, anchor=north east] at (0.88, -0.12) {$\z_C$};
		\node[figgold!80!black, font=\scriptsize, anchor=north west] at (1.0, -0.08) {$\z^\star\in\Z$};

		\node[figblue, font=\small] at (-1.5, 2.2) {$C$};
		\node[figgold!80!black, font=\small, anchor=west] at (3.2, 1.8) {$C'$};
		\draw[->, figgold!80!black, thin, shorten >=3pt] (3.15, 1.75) -- (2.60, 1.25);

		\node[figred!50, font=\scriptsize, anchor=south] at (-2, 2.9) {$\Z\times\R$};

		\node[black!65, font=\scriptsize, anchor=north] at (0.6, -2.7)
			{$C'=(1{-}\epsilon)C+\epsilon\w^\star,\qquad \epsilon = \tfrac{1}{2k}$};
	\end{tikzpicture}
\caption{Centroid rounding in Lemma~\ref{lem:centroid-rounding} (shown for $n=d=1$). The centroid $\mathfrak{c}(C)$ has $\z$ projection $\z_C\notin\Z^n$, and shrinking toward a point $\w^\star\in C$ with $\epsilon=1/(2k)$ yields $C'=(1-\epsilon)C+\epsilon\w^\star$. By affine equivariance, $\mathfrak{c}(C')=(1-\epsilon)\mathfrak{c}(C)+\epsilon\w^\star$, and $\w^\star$ is chosen so that $\proj_{\R^n}(\mathfrak{c}(C'))=\z^\star\in\Z^n$. The vertical lines indicate the fibers $\Z\times\R$.}
	\label{fig:centroid-rounding}
\end{figure}
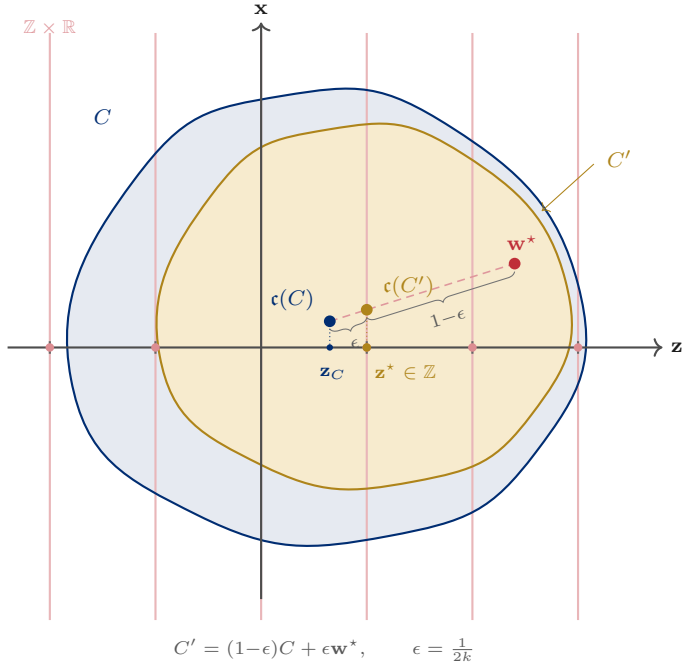

The next lemma ensures that after cutting a convex body by a hyperplane through a specified point, at least one side remains ``geometrically large'' in the sense that it contains a half-scale copy of the original body. This is exactly what is needed to apply Lemma~\ref{lem:disc-cont} to one of the pieces.

\begin{lemma}[Halfspace bisection]\label{lem:half-bisect}
	Let $D\subset \R^{n+d}$ be a convex body and let $H$ and $\bar H$ be two complementary closed halfspaces sharing the same boundary hyperplane.
	Then there exists $\mathbf{t}\in \R^{n+d}$ such that either
	\[
		\mathbf{t}+\frac12 D\subseteq D\cap H\qquad\text{or}\qquad \mathbf{t}+\frac12 D\subseteq D\cap \bar H.
	\]
\end{lemma}
\begin{proof}
	Let $H=\set{\y:\ \a^\T \y\ge c}$ and $\bar H=\set{\y:\ \a^\T \y\le c}$ with $\a\neq \0$.
	Choose points $\y^+,\y^-\in D$ attaining the maximum and minimum of $\a^\T \y$ over $D$.
	Define the two translates of $\frac12 D$ by
	\[
		D^+:=\frac12 D+\frac12 \y^+,\qquad D^-:=\frac12 D+\frac12 \y^-.
	\]
	Since $D$ is convex, for any $\y\in D$ we have $\frac12 \y+\frac12 \y^+\in D$, hence $D^+\subseteq D$.
	Similarly, $D^-\subseteq D$.

	Let $\theta:=\frac12\left(\a^\T \y^++\a^\T \y^-\right)$.
	If $c\le \theta$, we claim that $D^+\subseteq H$. Indeed, for $\y\in D$,
	\[
		\a^\T\Bigl(\frac12 \y+\frac12 \y^+\Bigr) = \frac12 \a^\T \y+\frac12 \a^\T \y^+ \ge \frac12 \a^\T \y^-+\frac12 \a^\T \y^+ = \theta\ge c,
	\]
	so $\frac12 \y+\frac12 \y^+\in H$ and hence $D^+\subseteq D\cap H$.

	If $c\ge \theta$, a symmetric argument shows $D^-\subseteq \bar H$, hence $D^-\subseteq D\cap \bar H$.
	In both cases we obtain the desired translate of $\frac12 D$.
\end{proof}

Figure~\ref{fig:halfspace-bisection-schematic} gives a visualization for Lemma~\ref{lem:half-bisect}. The same body $D$ is cut by a family of parallel hyperplanes orthogonal to $\a$, and the half scale translates $D^+$ and $D^-$ indicate which side is guaranteed to remain large when $c$ moves relative to $\theta$.

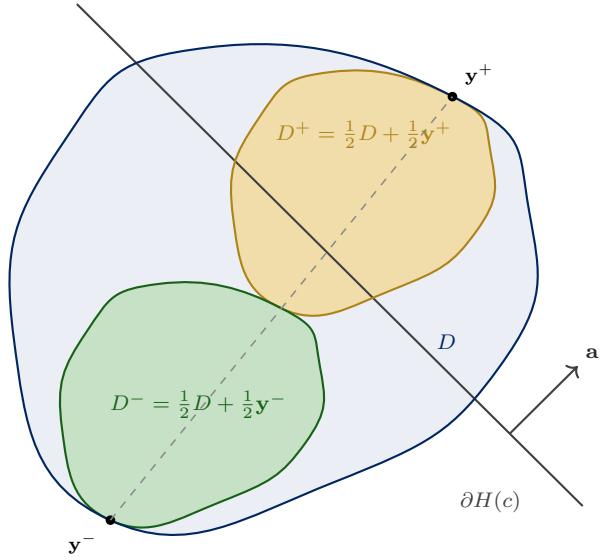
\begin{figure}[t]
	\centering
	\begin{tikzpicture}[x=1.28cm, y=1.28cm]
		\fill[figblue!8]
			plot[smooth cycle, tension=0.88] coordinates {
				(-2.5,-0.5) (-2.0,1.5) (-0.5,2.5) (2.0,2.0) (2.8,0.8) (2.5,-0.8) (0.8,-2.0) (-1.5,-2.4)
			};

		\begin{scope}[shift={(1.0055,0.9954)}, scale=0.5]
			\fill[figgold!38]
				plot[smooth cycle, tension=0.88] coordinates {
					(-2.5,-0.5) (-2.0,1.5) (-0.5,2.5) (2.0,2.0) (2.8,0.8) (2.5,-0.8) (0.8,-2.0) (-1.5,-2.4)
				};
		\end{scope}

		\begin{scope}[shift={(-0.7525,-1.1981)}, scale=0.5]
			\fill[figgreen!26]
				plot[smooth cycle, tension=0.88] coordinates {
					(-2.5,-0.5) (-2.0,1.5) (-0.5,2.5) (2.0,2.0) (2.8,0.8) (2.5,-0.8) (0.8,-2.0) (-1.5,-2.4)
				};
		\end{scope}

		\draw[black!75, thick] (-1.85,2.95) -- (3.35,-2.25);

		\begin{scope}[shift={(1.0055,0.9954)}, scale=0.5]
			\draw[figgold!80!black, thick]
				plot[smooth cycle, tension=0.88] coordinates {
					(-2.5,-0.5) (-2.0,1.5) (-0.5,2.5) (2.0,2.0) (2.8,0.8) (2.5,-0.8) (0.8,-2.0) (-1.5,-2.4)
				};
		\end{scope}
		\begin{scope}[shift={(-0.7525,-1.1981)}, scale=0.5]
			\draw[figgreen!70!black, thick]
				plot[smooth cycle, tension=0.88] coordinates {
					(-2.5,-0.5) (-2.0,1.5) (-0.5,2.5) (2.0,2.0) (2.8,0.8) (2.5,-0.8) (0.8,-2.0) (-1.5,-2.4)
				};
		\end{scope}

		\draw[figblue!85!black, thick]
			plot[smooth cycle, tension=0.88] coordinates {
				(-2.5,-0.5) (-2.0,1.5) (-0.5,2.5) (2.0,2.0) (2.8,0.8) (2.5,-0.8) (0.8,-2.0) (-1.5,-2.4)
			};

		\node[black!75, anchor=east, font=\small] at (2.8, -2.2) {$\partial H(c)$};

		\draw[->, thick, black!75] (2.6, -1.5) -- (3.3, -0.8);
		\node[anchor=south west, font=\small, text=black!85] at (3.3, -0.8) {$\a$};

		\fill[black] ( 2.0110, 1.9909) circle (1.8pt);
		\fill[black] (-1.5050,-2.3961) circle (1.8pt);
		\draw[black!45, semithick, dashed] (-1.5050,-2.3961) -- (2.0110,1.9909);
		\node[font=\scriptsize, anchor=south west] at ( 2.0600, 2.0350) {$\y^+$};
		\node[font=\scriptsize, anchor=north east] at (-1.5550,-2.4450) {$\y^-$};

		\node[font=\small, text=figblue!85!black] at (1.95,-0.55) {$D$};

		\node[figgold!80!black,  font=\small] at ( 1.1, 1.6) {$D^+=\tfrac12 D+\tfrac12\y^+$};
		\node[figgreen!70!black, font=\small] at (-0.6,-1.2) {$D^-=\tfrac12 D+\tfrac12\y^-$};

	\end{tikzpicture}
	\caption{Schematic for Lemma~\ref{lem:half-bisect}. A hyperplane $\partial H(c)$ orthogonal to $\a$ cuts $D$. The sets $D^+=\tfrac12D+\tfrac12\y^+$ and $D^-=\tfrac12D+\tfrac12\y^-$ are the two half scale translates that correspond to the alternatives in the lemma.}
	\label{fig:halfspace-bisection-schematic}
\end{figure}

\subsection{Proof of Theorem~\ref{thm:main}}

We begin by applying Lemma~\ref{lem:centroid-rounding} to shrink $C$ toward a suitable point $\w^\star$, producing a smaller body $C'\subseteq C$ whose centroid $\y^\star:=\mathfrak{c}(C')$ has an integer $\z$ component and therefore lies in $S$. For any closed halfspace $H$ containing $\y^\star$, Grünbaum's inequality bounds the $(n+d)$ dimensional volumes of $C'\cap H$ and $C'\cap \bar H$. Meanwhile, Lemma~\ref{lem:half-bisect} guarantees that one of these two pieces contains a translate of $\frac12 C'$, which in turn contains a sufficiently large box in its projection and hence satisfies the hypotheses of the comparison inequality (Lemma~\ref{lem:disc-cont}). We then convert between continuous and discrete volumes using Lemma~\ref{lem:disc-cont} and combine the bounds to obtain \eqref{eq:main-bound}.

\begin{proof}[of Theorem~\ref{thm:main}]
	Assume the hypotheses of Theorem~\ref{thm:main}.
	Since $k\ge \frac{3e}{2}(n+d)$, we have $k>1/2$, so Lemmas~\ref{lem:disc-cont} and~\ref{lem:centroid-rounding} apply.

	We first construct a point $\y^\star\in S$ with an integer projection.
	Let $\epsilon:=1/(2k)$. By Lemma~\ref{lem:centroid-rounding}, there exist $\w^\star\in C$ and a set
	\[
		C':=(1-\epsilon)C+\epsilon \w^\star \subseteq C
	\]
	such that $\proj_{\R^n}(\mathfrak{c}(C'))\in \Z^n$ and $\proj_{\R^n}(C')$ contains a translate of
	$(k-\tfrac12)\mathbb{B}_\infty^n$.
	Define
	\[
		\y^\star:=\mathfrak{c}(C').
	\]
	Then $\y^\star\in C'\subseteq C$, and its $\z$ component is an integer vector, hence $\y^\star\in \Lambda$.
	Therefore, $\y^\star\in S$.

	Let $H$ be a closed halfspace with $\y^\star\in H$.
	By Lemma~\ref{lem:boundary-reduction}, it suffices to prove \eqref{eq:main-bound} under the additional assumption $\y^\star\in \partial H$.
	Fix such an $H$ and write it as
	\[
		H=\set{\y:\ \a^\T(\y-\y^\star)\ge 0}\qquad (\a\neq \0).
	\]
		Let $\bar H=\set{\y:\ \a^\T(\y-\y^\star)\le 0}$ be the complementary closed halfspace.

		\[
			A:=C'\cap H,\qquad B:=C'\cap \bar H.
		\]

	Since $\y^\star=\mathfrak{c}(C')\in H\cap \bar H$, Grünbaum's inequality (Theorem~\ref{thm:grunbaum}) gives
	\begin{equation}\label{eq:grunbaum-AB}
		\vol_{n+d}(A)\ge \frac{1}{e}\vol_{n+d}(C'),\qquad \vol_{n+d}(B)\ge \frac{1}{e}\vol_{n+d}(C').
	\end{equation}
	Moreover, $C'=A\cup B$ and $A\cap B=C'\cap \partial H$ is contained in the boundary hyperplane $\partial H$, which has $(n+d)$ dimensional Lebesgue measure $0$.
	Therefore, $\vol_{n+d}(C')=\vol_{n+d}(A)+\vol_{n+d}(B)$, and we obtain
	\begin{equation}\label{eq:volB-upper}
		\vol_{n+d}(B) = \vol_{n+d}(C')-\vol_{n+d}(A) \le \Bigl(1-\frac{1}{e}\Bigr)\vol_{n+d}(C').
	\end{equation}

	Next, Lemma~\ref{lem:half-bisect} applied to $D=C'$ and the split $(H,\bar H)$ shows that one of $A$ or $B$
	contains a translate of $\frac12 C'$. Since $\proj_{\R^n}(C')$ contains a translate of $(k-\tfrac12)\mathbb{B}_\infty^n$, this implies that one of $\proj_{\R^n}(A)$ or $\proj_{\R^n}(B)$ contains a translate of $\frac{k-\tfrac12}{2}\mathbb{B}_\infty^n$.
	Since $k>3/2$, we have $\frac{k-\tfrac12}{2}>1/2$, so Lemma~\ref{lem:disc-cont} is applicable in either case.

	We write $S':=C'\cap \Lambda$. Since $C'\subseteq C$, we have $S'\subseteq S$ and hence
	\begin{equation}\label{eq:S-vs-Sprime}
		\H_d(S\cap H)\ge \H_d(S'\cap H).
	\end{equation}

	We now distinguish the two cases, depending on whether the large projection occurs on $A$ or on $B$.

	\noindent\textbf{Case 1: the large projection occurs on $A$.}
	Since $\proj_{\R^n}(A)$ contains a translate of $\frac{k-\tfrac12}{2}\mathbb{B}_\infty^n$, Lemma~\ref{lem:disc-cont} yields
	\[
		\H_d(A\cap \Lambda)\ge \Bigl(1-\frac{1}{k-\tfrac12}\Bigr)^{n+d} \vol_{n+d}(A)=\Bigl(\frac{2k-3}{2k-1}\Bigr)^{n+d}\vol_{n+d}(A).
	\]
	Since $S'=C'\cap \Lambda$ and $A=C'\cap H$, we have $S'\cap H=A\cap \Lambda$.
	Therefore, $\H_d(S'\cap H)=\H_d(A\cap \Lambda)$, and hence, using \eqref{eq:grunbaum-AB},
	\begin{equation}\label{eq:case1-volCprime}
		\H_d(S'\cap H) \ge \frac{1}{e}\Bigl(\frac{2k-3}{2k-1}\Bigr)^{n+d}\vol_{n+d}(C').
	\end{equation}

	We relate $\vol_{n+d}(C')$ to $\H_d(S)$. Applying Lemma~\ref{lem:disc-cont} to $C$ with parameter $k$ gives
	\[
		\H_d(S)\le \Bigl(1+\frac{1}{2k}\Bigr)^{n+d} \vol_{n+d}(C),\qquad\text{so}\qquad \vol_{n+d}(C)\ge \Bigl(\frac{2k}{2k+1}\Bigr)^{n+d} \H_d(S).
	\]
	Since $\vol_{n+d}(C')=(1-\frac{1}{2k})^{n+d}\vol_{n+d}(C)$, we obtain
	\begin{equation}\label{eq:volCprime-vs-nuS}
		\vol_{n+d}(C') \ge \Bigl(\frac{2k-1}{2k}\Bigr)^{n+d}\Bigl(\frac{2k}{2k+1}\Bigr)^{n+d} \H_d(S) = \Bigl(\frac{2k-1}{2k+1}\Bigr)^{n+d} \H_d(S).
	\end{equation}
	Combining \eqref{eq:case1-volCprime}, \eqref{eq:volCprime-vs-nuS}, and \eqref{eq:S-vs-Sprime} gives
	\begin{equation}\label{eq:case1-final}
		\H_d(S\cap H) \ge \frac{1}{e}\Bigl(\frac{2k-3}{2k+1}\Bigr)^{n+d} \H_d(S).
	\end{equation}

		\noindent\textbf{Case 2: the large projection occurs on $B$.}

		Since $S'=C'\cap \Lambda$, $B\cap \Lambda=(C'\cap \bar H)\cap \Lambda=S'\cap \bar H$, and $H\cup \bar H=\R^{n+d}$, we have
			$S'=(S'\cap H)\cup(S'\cap \bar H)=(S'\cap H)\cup(B\cap \Lambda)$.
		By subadditivity of $\H_d$,
		\begin{equation}\label{eq:nu-lower-by-B}
			\H_d(S'\cap H)\ge \H_d(S')-\H_d(B\cap \Lambda).
		\end{equation}

	We now bound the two terms on the right hand side.
	First, since $\proj_{\R^n}(C')$ contains a translate of $(k-\tfrac12)\mathbb{B}_\infty^n$, Lemma~\ref{lem:disc-cont} yields
	\begin{equation}\label{eq:nuSprime-lower}
		\H_d(S')\ge \Bigl(1-\frac{1}{2k-1}\Bigr)^{n+d} \vol_{n+d}(C') = \Bigl(\frac{2k-2}{2k-1}\Bigr)^{n+d}\vol_{n+d}(C').
	\end{equation}
	Second, since $\proj_{\R^n}(B)$ contains a translate of $\frac{k-\tfrac12}{2}\mathbb{B}_\infty^n$, Lemma~\ref{lem:disc-cont} yields
	\[
		\H_d(B\cap \Lambda)\le \Bigl(1+\frac{1}{k-\tfrac12}\Bigr)^{n+d} \vol_{n+d}(B) = \Bigl(\frac{2k+1}{2k-1}\Bigr)^{n+d} \vol_{n+d}(B).
	\]
	Using \eqref{eq:volB-upper}, we obtain
	\begin{equation}\label{eq:nuB-upper}
		\H_d(B\cap \Lambda) \le \Bigl(1-\frac{1}{e}\Bigr)\Bigl(\frac{2k+1}{2k-1}\Bigr)^{n+d} \vol_{n+d}(C').
	\end{equation}
	Substituting \eqref{eq:nuSprime-lower} and \eqref{eq:nuB-upper} into \eqref{eq:nu-lower-by-B} gives
	\[
		\H_d(S'\cap H) \ge \left[
			\Bigl(\frac{2k-2}{2k-1}\Bigr)^{n+d}-\Bigl(1-\frac{1}{e}\Bigr)\Bigl(\frac{2k+1}{2k-1}\Bigr)^{n+d}
			\right]\vol_{n+d}(C').
	\]
	Using \eqref{eq:volCprime-vs-nuS} and $S'\subseteq S$ yields
	\begin{equation}\label{eq:case2-final}
		\H_d(S\cap H) \ge \left[
			\Bigl(\frac{2k-2}{2k+1}\Bigr)^{n+d}-\Bigl(1-\frac{1}{e}\Bigr)
			\right]\H_d(S).
	\end{equation}

	\noindent\textbf{A uniform bound.}
	From \eqref{eq:case1-final} and \eqref{eq:case2-final}, for every closed halfspace $H$ with $\y^\star\in \partial H$,
	\[
		\frac{\H_d(S\cap H)}{\H_d(S)} \ge \min\left\{
		\frac{1}{e}\Bigl(\frac{2k-3}{2k+1}\Bigr)^{n+d},\
		\Bigl(\frac{2k-2}{2k+1}\Bigr)^{n+d}-\Bigl(1-\frac{1}{e}\Bigr)
		\right\}.
	\]

	We bound both terms from below using Bernoulli's inequality $(1-t)^{n+d}\ge 1-(n+d)t$ for $t\in [0,1]$.

	For the first term,
	\[
		\Bigl(\frac{2k-3}{2k+1}\Bigr)^{n+d} = \Bigl(1-\frac{4}{2k+1}\Bigr)^{n+d} \ge 1-\frac{4(n+d)}{2k+1} \ge 1-\frac{2(n+d)}{k},
	\]
	and hence
	\[
		\frac{1}{e}\Bigl(\frac{2k-3}{2k+1}\Bigr)^{n+d} \ge \frac{1}{e}-\frac{2(n+d)}{ek}\ge \frac{1}{e}-\frac{3(n+d)}{2k},
	\]
	where the last inequality uses $\frac{2}{e}\le \frac{3}{2}$.

	For the second term,
	\[
		\Bigl(\frac{2k-2}{2k+1}\Bigr)^{n+d} = \Bigl(1-\frac{3}{2k+1}\Bigr)^{n+d} \ge 1-\frac{3(n+d)}{2k+1} \ge 1-\frac{3(n+d)}{2k},
	\]
	and therefore
	\[
		\Bigl(\frac{2k-2}{2k+1}\Bigr)^{n+d}-\Bigl(1-\frac{1}{e}\Bigr) \ge \frac{1}{e}-\frac{3(n+d)}{2k}.
	\]

	Combining the two bounds shows that for every closed halfspace $H$ with $\y^\star\in \partial H$,
	\[
		\H_d(S\cap H)\ge \Bigl(\frac{1}{e}-\frac{3(n+d)}{2k}\Bigr)\H_d(S).
	\]
	By Lemma~\ref{lem:boundary-reduction}, the same inequality holds for every closed halfspace $H$ with $\y^\star\in H$.
	This proves \eqref{eq:main-bound} and completes the proof of Theorem~\ref{thm:main}.
\end{proof}

\subsection{Proofs of Corollaries~\ref{cor:main} and~\ref{cor:unimodular}}
\begin{proof}[of Corollary~\ref{cor:main}]

	Let $\y^\star\in S$ be the point given by Theorem~\ref{thm:main}.
	Since $k\ge 3e(n+d)$, we have $\frac{1}{e}-\frac{3(n+d)}{2k}\ge \frac{1}{2e}$, and therefore Theorem~\ref{thm:main} implies
	\[
		\H_d(S\cap H)\ge \frac{1}{2e}\,\H_d(S)
	\]
	for every closed halfspace $H\subseteq \R^{n+d}$ with $\y^\star\in H$, and hence $h_S(\y^\star)\ge \frac{1}{2e}\,\H_d(S)$.

	Since $\y^\star\in S$, we have $S\neq \emptyset$.
	If $\mathcal{C}(S)\neq \emptyset$, then for any $\hat \y\in \mathcal{C}(S)$ we have
	\[
		h_S(\hat \y)\ge h_S(\y^\star)\ge \frac{1}{2e}\,\H_d(S).
	\]
	Since $n\ge 1$ implies $2^n\ge 2$, we also have $\frac{1}{2e}\ge \frac{1}{2^n e}$, which yields \eqref{eq:centerpoint-bound}.
	Thus, it remains to show that $\mathcal{C}(S)\neq \emptyset$.

	For a unit vector $\u\in \R^{n+d}$ and $t\in \R$, define the closed halfspace
	\[
		H(\u,t):=\set{\y\in \R^{n+d}:\ \u^\T \y \ge t}.
	\]

	The set $S=C\cap \Lambda$ is compact because $C$ is compact and $\Lambda$ is closed.

	Define the set function $\mu(A):=\H_d(S\cap A)$ for measurable $A\subseteq \R^{n+d}$.
	Then $\mu$ is a finite measure because $\H_d(S)<\infty$.

	Fix $\u\in \mathbb{S}^{n+d-1}$ and define $F_{\u}(t):=\mu(H(\u,t))=\H_d(S\cap H(\u,t))$.
	The map $F_{\u}$ is nonincreasing.
	Moreover, since $H(\u,t-\epsilon)\downarrow H(\u,t)$ as $\epsilon\downarrow 0$, continuity from above gives $F_{\u}(t-\epsilon)\downarrow F_{\u}(t)$.
	Therefore, if $t_k\to t$ and $\epsilon>0$, then for all sufficiently large $k$ we have $t_k\ge t-\epsilon$ and hence
	$F_{\u}(t_k)\le F_{\u}(t-\epsilon)$.
	Taking $\limsup_{k\to\infty}$ and then letting $\epsilon\downarrow 0$ yields $\limsup_{k\to\infty}F_{\u}(t_k)\le F_{\u}(t)$, so $F_{\u}$ is upper semicontinuous.

	Define $\phi_{\u}(\y):=F_{\u}(\u^\T \y)=\H_d\bigl(S\cap H(\u,\u^\T \y)\bigr)$.
	Since $\y\mapsto \u^\T \y$ is continuous, $\phi_{\u}$ is upper semicontinuous.

	By Lemma~\ref{lem:boundary-reduction}, for $\y\in S$,
	\[
		h_S(\y)=\inf_{\u\in \mathbb{S}^{n+d-1}} \phi_{\u}(\y),
	\]
	so $h_S$ is upper semicontinuous.
	An upper semicontinuous function on a compact set attains its maximum, so $\argmax_{\y\in S} h_S(\y)$ is nonempty, hence $\mathcal{C}(S)\neq \emptyset$.
\end{proof}

\begin{proof}[of Corollary~\ref{cor:unimodular}]
	Consider the invertible linear map $T_U:\R^{n+d}\to\R^{n+d}$ defined by
	\[
		T_U(\z,\x):=(U^{-1}\z,\,\x),\qquad T_U^{-1}(\z,\x)=(U\z,\,\x).
	\]
	Since $|\det U|=1$, Cramer's rule gives $U^{-1}\in\Z^{n\times n}$, so $T_U(\Lambda)=\Lambda$.

	Set $\widetilde C:=T_U(C)$, $\widetilde S:=\widetilde C\cap\Lambda$, and $\widetilde K:=\proj_{\R^n}(\widetilde C)=U^{-1}K$.
	Applying $U^{-1}$ to \eqref{eq:big-box-unimod} gives
	\[
		U^{-1}\z_0+k\,\mathbb{B}_\infty^n
		= U^{-1}\bigl(\z_0+k\,U\mathbb{B}_\infty^n\bigr)
		\subseteq U^{-1}K = \widetilde K,
	\]
	so $\widetilde C$ satisfies the conditions \eqref{eq:big-box} of Theorem~\ref{thm:main} with the same $k$.
	By Theorem~\ref{thm:main}, there exists $\widetilde\y^\star\in\widetilde S$ such that for every closed halfspace $\widetilde H$ with $\widetilde\y^\star\in\widetilde H$,
	\[
		\H_d(\widetilde S\cap\widetilde H)\ge\Bigl(\tfrac{1}{e}-\tfrac{3(n+d)}{2k}\Bigr)\H_d(\widetilde S).
	\]

	Set $\y^\star:=T_U^{-1}(\widetilde\y^\star)$. Since $\widetilde\y^\star\in\widetilde C=T_U(C)$, we have $\y^\star\in C$. Since $T_U^{-1}(\Lambda)=\Lambda$ and $\widetilde\y^\star\in\Lambda$, we have $\y^\star\in\Lambda$. Hence $\y^\star\in S$.

		For every measurable set $A\subseteq\Lambda$ and each $\z\in\Z^n$, we have $(T_U(A))_\z=A_{U\z}$, since $(\z,\x)\in T_U(A)$ if and only if $(U\z,\x)\in A$. Because $\z\mapsto U\z$ is a bijection of $\Z^n$,
		\[
			\H_d(T_U(A))
			=\sum_{\z\in\Z^n}\vol_d(A_{U\z})
			=\sum_{\z\in\Z^n}\vol_d(A_\z)
			=\H_d(A).
		\]
		Fix any closed halfspace $H$ with $\y^\star\in H$, and set $\widetilde H:=T_U(H)$.
		Then $\widetilde H$ is a closed halfspace containing $\widetilde\y^\star=T_U(\y^\star)$, and we have
		$T_U(S)=\widetilde S$ and $T_U(S\cap H)=\widetilde S\cap\widetilde H$.
		Therefore,
		\[
			\begin{aligned}
				\H_d(S\cap H)
				 & =\H_d(T_U(S\cap H))
				=\H_d(\widetilde S\cap\widetilde H)\\
				&\ge \Bigl(\tfrac{1}{e}-\tfrac{3(n+d)}{2k}\Bigr)\H_d(\widetilde S)
				=\Bigl(\tfrac{1}{e}-\tfrac{3(n+d)}{2k}\Bigr)\H_d(S).
			\end{aligned}
		\]
		This yields the desired inequality for $H$.
\end{proof}

\section{Sharpness of the linear threshold}\label{sec:necessity}

This section shows that the linear scaling in Theorem~\ref{thm:main} cannot, in general, be improved if one only assumes that $\proj_{\R^n}(C)$ contains a translate of $k\mathbb{B}_\infty^n$.
We construct a family of convex bodies whose mixed-integer volume is concentrated on $2^n$ ``central'' integer fibers, while for every point in $S$ one can choose a halfspace containing it that intersects this central family in at most one fiber.
Choosing the continuous dimension $d$ large relative to $kn$ makes the remaining fibers negligible, leading to an exponentially small depth ratio of order $\exp(-\Theta((n+d)/k))$. Thus, if $k$ is sublinear in $n+d$, then we will have exponentially small halfspace depth for any point in $S$.

\begin{theorem}[Necessity of linear $\ell_\infty$ radius]\label{thm:linear-necessity-total-dim}
	For any  constant $\alpha>\log 2$, there exist constants $\gamma_\alpha,M_\alpha>0$ (depending only on $\alpha$) with the following property.

	For every integer $k\ge 2$ and every integer $m\ge 2(1+\alpha k)$, there exist integers $n,d\ge 1$ with $n+d=m$ and a convex body
	$C\subset\R^{n+d}$ such that its projection $K:=\proj_{\R^n}(C)$ contains a translate of $k\mathbb{B}_\infty^n$ and, writing $S:=C\cap(\Z^n\times\R^d)$, one has the quantitative upper bound
	\begin{equation}\label{eq:exp-upper}
		\sup_{\y\in S}\frac{h_S(\y)}{\H_d(S)}\ \le\ M_\alpha\exp\!\Bigl(-\gamma_\alpha\,\frac{m}{k}\Bigr).
	\end{equation}
\end{theorem}

Figure~\ref{fig:necessity-construction} illustrates this construction for $n=d=1$ with $k=3$.

\begin{figure}[t]
	\centering
		\begin{tikzpicture}[x=1.1cm, y=1.8cm, scale=1.1, transform shape]

		\fill[figgreen!13] (4.0, -1.22) rectangle (7.6, 1.52);

		\fill[figblue!10] (0.5, 0) -- (3.5, 1.0) -- (6.5, 0) -- (3.5, -1.0) -- cycle;

		\foreach \zz in {1,2,3,4,5,6} {
			\draw[black!14, densely dotted] (\zz, -1.20) -- (\zz, 0.98);
		}

		\draw[figred!45, line width=1.1pt] (1.00, -0.167) -- (1.00, 0.167);
		\draw[figred!45, line width=1.1pt] (2.00, -0.500) -- (2.00, 0.500);
		\draw[figred!45, line width=1.1pt] (5.00, -0.500) -- (5.00, 0.500);
		\draw[figred!45, line width=1.1pt] (6.00, -0.167) -- (6.00, 0.167);
		\foreach \zz in {1,2,5,6} { \fill[figred!50] (\zz, 0) circle (1.3pt); }

		\draw[figred, line width=1.9pt] (3.00, -0.833) -- (3.00, 0.833);
		\draw[figred, line width=1.9pt] (4.00, -0.833) -- (4.00, 0.833);
		\foreach \zz in {3,4} { \fill[figred] (\zz, 0) circle (1.6pt); }

		\draw[figblue, thick] (0.5, 0) -- (3.5, 1.0) -- (6.5, 0) -- (3.5, -1.0) -- cycle;

		\draw[figgreen!70!black, line width=1.2pt, dashed] (4.0, -1.20) -- (4.0, 1.50);

		\draw[figred, line width=1.9pt] (4.00, -0.833) -- (4.00, 0.833);
		\node[figgreen!70!black, font=\scriptsize, anchor=west] at (4.16, -1.10) {$\partial H$};

		\draw[->, thick, black!70] (-0.3, 0) -- (7.7, 0)
			node[right, font=\small, text=black!90] {$\z$};
		\draw[->, thick, black!70] (0.0, -1.15) -- (0.0, 1.65)
			node[above, font=\small, text=black!90] {$\x$};

		\foreach \zz in {1,2,3,4,5,6} {
			\draw[black!60] (\zz, 0.05) -- (\zz, -0.05);
		}
		\node[below=2pt, font=\scriptsize, black!75] at (1, -0.04) {$1$};
		\node[below=2pt, font=\scriptsize, black!75] at (2, -0.04) {$2$};
		\node[below=2pt, font=\scriptsize, black!75] at (2.90, -0.04) {$k$};
		\node[below=2pt, font=\scriptsize, black!75] at (4.10, -0.04) {$k{+}1$};
		\node[below=2pt, font=\scriptsize, black!75] at (5, -0.04) {$2k{-}1$};
		\node[below=2pt, font=\scriptsize, black!75] at (6, -0.04) {$2k$};

		\draw[black!30, densely dotted] (0.5, 0.02) -- (0.5, -1.20);
		\draw[black!30, densely dotted] (6.5, 0.02) -- (6.5, -1.20);
		\draw[<->, black!40, thin] (0.5, -1.18) -- (6.5, -1.18)
			node[midway, below=7pt, xshift=-8pt, font=\scriptsize, black!55]
			{$K=\bigl[\tfrac{1}{2},\;2k+\tfrac{1}{2}\bigr]$};

		\node[figred!75, anchor=south, font=\scriptsize] at (3.5, 1.47)
			{$V_0=\{k,\;k{+}1\}$};
		\draw[->, figred!55, thin, shorten >=2pt]
			(3.28, 1.45) -- (3.04, 0.84);
		\draw[->, figred!55, thin, shorten >=2pt]
			(3.72, 1.45) -- (3.96, 0.84);

		\draw[<->, black!50, thin] (2.84, 0) -- (2.84, 0.833);
		\node[left=2pt, font=\scriptsize, black!60] at (2.82, 0.42)
			{$t(k)=1-\tfrac{1}{2k}$};

		\node[figgreen!65!black, anchor=south, font=\small] at (5.9, 1.20)
			{$H=\{\z\ge k{+}1\}$};
		\node[figgreen!65!black, font=\scriptsize, anchor=east] at (3.66, 0.5) {$\notin H$};
		\node[figgreen!65!black, font=\scriptsize, anchor=west] at (3.96, 0.5) {$\in H$};

		\node[figblue, font=\small] at (0.95, 0.65) {$C$};

		\fill[figred] (4.0, 0.26) circle (2.0pt);
		\node[figred, font=\small, anchor=west] at (4.13, 0.28) {$\y$};
	\end{tikzpicture}
	\caption{Sharpness construction in Theorem~\ref{thm:linear-necessity-total-dim} for $n=d=1$ with $k=3$. The body $C$ (blue) has projection $K=[\tfrac{1}{2},2k+\tfrac{1}{2}]$ onto the $\z$ axis, and the mixed-integer set $S=C\cap(\Z\times\R)$ consists of the fiber intervals (red segments) over $\z\in K\cap\Z$. The two central fibers $V_0=\{k,k+1\}$ (thicker segments) have maximal radius $t(k)=1-\tfrac{1}{2k}$. The halfspace $H=\{\z\ge k+1\}$ (green) contains the fiber at $\z=k+1$ and excludes the equally large fiber at $\z=k$. Because $H$ is closed, the boundary fiber at $z=k+1$ is included.}
	\label{fig:necessity-construction}
\end{figure}
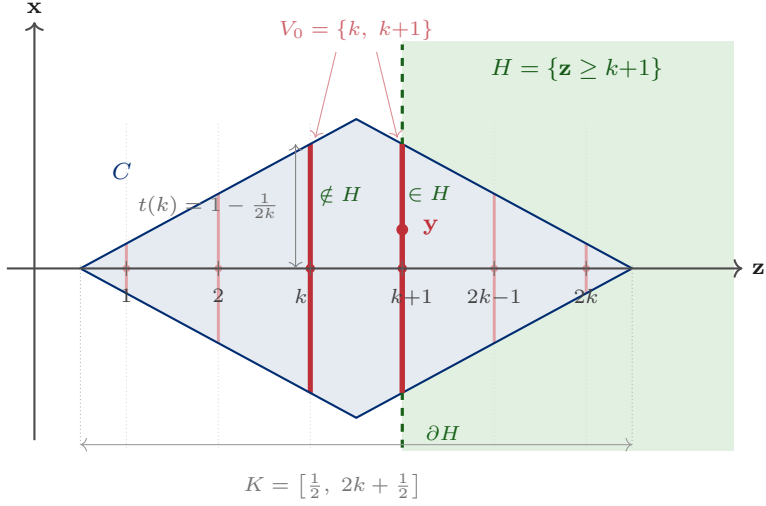

\begin{proof}
	Fix $\alpha>\log 2$, an integer $k\ge 2$, and an integer $m\ge 2(1+\alpha k)$.
	Define
	\[
		n:=\Bigl\lfloor \frac{m}{1+\alpha k}\Bigr\rfloor,\qquad d:=m-n.
	\]
	Then $n\ge 2$ and $d\ge 1$.
	In fact, $(1+\alpha k)n\le m$ implies $d=m-n\ge \alpha k n$.

	Define the center $\c:=(k+\tfrac12)(1,\dots,1)\in\R^n$ and the box
	\[
		K:=\c+k\mathbb{B}_\infty^n=[\tfrac12,\,2k+\tfrac12]^n.
	\]
	Define the function $t:K\to[0,1]$ by
	\[
		t(\z):=1-\frac{1}{k}\|\z-\c\|_\infty.
	\]
	Finally, define the convex body $C\subset\R^{n+d}$ by
	\[
		C:=\set{(\z,\x)\in\R^n\times\R^d:\ \z\in K,\ \|\x\|_2\le t(\z)}.
	\]
	Since $\|\cdot\|_\infty$ is convex, the map $\z\mapsto -\|\z-\c\|_\infty$ is concave, hence $t$ is concave.
	Thus $C$ is convex: if $(\z_i,\x_i)\in C$ and $\lambda\in[0,1]$, then $\z_\lambda:=(1-\lambda)\z_1+\lambda\z_2\in K$ and
	\[
		\|\,(1-\lambda)\x_1+\lambda\x_2\,\|_2
		\le (1-\lambda)\|\x_1\|_2+\lambda\|\x_2\|_2
		\le (1-\lambda)t(\z_1)+\lambda t(\z_2)
		\le t(\z_\lambda),
	\]
	so $\bigl(\z_\lambda,(1-\lambda)\x_1+\lambda\x_2\bigr)\in C$.
	Moreover, since $K$ is compact and $t$ is continuous with $0\le t\le 1$, the set $C$ is compact, and since $t(\c)=1$ and $\c\in \intset(K)$, the set $C$ has nonempty interior.
	Clearly $\proj_{\R^n}(C)=K$, so $K$ contains a translate of $k\mathbb{B}_\infty^n$.

	Let $S:=C\cap(\Z^n\times\R^d)$.
	Write $v_d:=\vol_d(\set{\x\in\R^d:\|\x\|_2\le 1})$.
	For each $\z\in K$, the fiber $C_\z:=\set{\x:(\z,\x)\in C}$ is a Euclidean ball in $\R^d$ of radius $t(\z)$, hence $\vol_d(C_\z)=v_d t(\z)^d$.
	Since $K\cap\Z^n=\set{1,2,\dots,2k}^n$, the mixed-integer volume is
	\[
		\H_d(S)=\sum_{\z\in K\cap\Z^n}\vol_d(C_\z)=v_d\sum_{\z\in K\cap\Z^n} t(\z)^d.
	\]

	Set
	\[
		V_0:=\set{k,k+1}^n\subset\Z^n,\qquad \#V_0=2^n,\qquad T:=\sum_{\substack{\z\in K\cap\Z^n\\\z\notin V_0}} t(\z)^d.
	\]
	For $\z\in\Z^n$, each coordinate $z_i-(k+\tfrac12)$ is an integer plus $\tfrac12$, so $\|\z-\c\|_\infty\in\{\tfrac12,\tfrac32,\tfrac52,\dots\}$.
	The points of $K\cap\Z^n$ with $\|\z-\c\|_\infty=\tfrac12$ are exactly $V_0$, and for all $\z\in V_0$ one has $t(\z)=1-\frac{1}{2k}$.
	Therefore,
	\begin{equation}\label{eq:HdS-split-proof}
		\H_d(S)=v_d\left(2^n\Bigl(1-\frac{1}{2k}\Bigr)^d+T\right).
	\end{equation}

	Fix $\y=(\z,\x)\in S$. We construct a closed halfspace $H\subset\R^{n+d}$ with $\y\in H$ such that $H$ intersects $V_0\times\R^d$ in at most one integer fiber.
	If $\z=\v\in V_0$, define $\sigma(\v)\in\{\pm 1\}^n$ by $\sigma(\v)_i:=2(\v_i-k)-1$ for $i\in\{1,\dots,n\}$ and let
	\[
		H:=H_\v:=\set{(\z,\x)\in\R^{n+d}:\ \sigma(\v)^\T \z \ge \sigma(\v)^\T \v}.
	\]
	Then $\y\in H$. Since $\sigma(\v)_i = +1$ when $\v_i = k+1$ and $\sigma(\v)_i = -1$ when $\v_i = k$, the point $\v$ is the unique maximizer of $\z\mapsto \sigma(\v)^\T\z$ over $V_0$. Hence $H$ intersects $V_0\times\R^d$ in exactly one fiber and
	\[
		\H_d(S\cap H)\le v_d\left(\Bigl(1-\frac{1}{2k}\Bigr)^d+T\right).
	\]
	If instead $\z\notin V_0$, then some coordinate $\z_i\notin\{k,k+1\}$, so $\z_i\le k-1$ or $\z_i\ge k+2$.
	In the first case take $H:=\set{(\z,\x)\in\R^{n+d}:\ \z_i\le k-1}$ and in the second take $H:=\set{(\z,\x)\in\R^{n+d}:\ \z_i\ge k+2}$.
	In either case $\y\in H$ and $H\cap(V_0\times\R^d)=\emptyset$, hence $\H_d(S\cap H)\le v_dT$.
	Since $h_S(\y)\le \H_d(S\cap H)$ by definition, it follows that for all $\y\in S$,
	\begin{equation}\label{eq:uniform-depth-upper}
		h_S(\y)\ \le\ v_d\left(\Bigl(1-\frac{1}{2k}\Bigr)^d+T\right).
	\end{equation}

	Combining \eqref{eq:HdS-split-proof} and \eqref{eq:uniform-depth-upper} yields
	\[
		\sup_{\y\in S}\frac{h_S(\y)}{\H_d(S)}
		\le
		\frac{\Bigl(1-\frac{1}{2k}\Bigr)^d+T}{2^n\Bigl(1-\frac{1}{2k}\Bigr)^d+T}
		\le
		2^{-n}+\frac{T}{2^n\Bigl(1-\frac{1}{2k}\Bigr)^d}.
	\]

	To bound $\frac{T}{2^n\left(1-\frac{1}{2k}\right)^d}$, for $j\in\{0,1,\dots,k-1\}$ define the shells
	\[
		V_j:=\Bigl\{\z\in K\cap\Z^n:\ \|\z-\c\|_\infty=\frac{2j+1}{2}\Bigr\}.
	\]
	Then $(K\cap\Z^n)\setminus V_0=\bigcup_{j=1}^{k-1}V_j$ and for $\z\in V_j$ one has
	$t(\z)=1-\frac{2j+1}{2k}$, hence
	\[
		\frac{t(\z)}{1-\frac{1}{2k}}
		=\frac{1-\frac{2j+1}{2k}}{1-\frac{1}{2k}}
		=\frac{2k-2j-1}{2k-1}
		=1-\frac{2j}{2k-1}.
	\]
	Moreover, $V_j$ is contained in the cube $\{\z\in\Z^n:\ \|\z-\c\|_\infty\le \frac{2j+1}{2}\}$, whose cardinality is at most $(2j+2)^n$.
	Therefore, $\#V_j/2^n\le (j+1)^n$.
	Using $1-x \leq e^{-x}$ for any $x \in \R$ and $d\ge \alpha k n$, we obtain
	\[
		\Bigl(\frac{t(\z)}{1-\frac{1}{2k}}\Bigr)^d
		\le \exp\!\Bigl(-\frac{2j}{2k-1}\,d\Bigr)
		\le \exp(-\alpha j n),
		\qquad \z\in V_j.
	\]
	Therefore,
	\[
		\frac{T}{2^n\left(1-\frac{1}{2k}\right)^d}
		= \sum_{j=1}^{k-1}\frac{\#V_j}{2^n}\left(1-\frac{2j}{2k-1}\right)^d
		\le \sum_{j=1}^{k-1}(j+1)^n e^{-\alpha j n}.
	\]
	Since $j+1\le 2^j$ for all $j\ge 1$, the summand is at most $2^{jn}e^{-\alpha j n}=e^{-(\alpha-\log 2)jn}$.
	Thus
	\begin{equation}\label{eq:shell-geometric}
		\frac{T}{2^n\left(1-\frac{1}{2k}\right)^d}
		\le
		\sum_{j=1}^{\infty}e^{-(\alpha-\log 2)jn}
		= \frac{e^{-(\alpha-\log 2)n}}{1-e^{-(\alpha-\log 2)n}}
		\le \frac{e^{-(\alpha-\log 2)n}}{1-e^{-(\alpha-\log 2)}}.
	\end{equation}
	It follows that
	\begin{equation}\label{eq:ratio-n-bound}
		\sup_{\y\in S}\frac{h_S(\y)}{\H_d(S)}
		\le
		2^{-n}+\frac{e^{-(\alpha-\log 2)n}}{1-e^{-(\alpha-\log 2)}}.
	\end{equation}

	To rewrite \eqref{eq:ratio-n-bound} in terms of $m$ and $k$, note that $m\ge 2(1+\alpha k)$ implies
	\[
		n=\Bigl\lfloor \frac{m}{1+\alpha k}\Bigr\rfloor \ge \frac{m}{1+\alpha k}-1\ge \frac{m}{2(1+\alpha k)}.
	\]
	Set $M_\alpha:=1+\frac{1}{1-e^{-(\alpha-\log 2)}}$.
	Then \eqref{eq:ratio-n-bound} gives
	\[
		\begin{aligned}
			\sup_{\y\in S}\frac{h_S(\y)}{\H_d(S)}
			&\le M_\alpha\exp\!\bigl(-\min\{\log 2,\ \alpha-\log 2\}\,n\bigr) \\
			&\le M_\alpha\exp\!\Bigl(-\min\{\log 2,\ \alpha-\log 2\}\,\frac{m}{2(1+\alpha k)}\Bigr).
		\end{aligned}
	\]
	Since $1+\alpha k\le (\alpha+1)k$ for $k\ge 1$, the right hand side is at most
	\[
		M_\alpha\exp\!\Bigl(-\frac{\min\{\log 2,\ \alpha-\log 2\}}{2(\alpha+1)}\,\frac{m}{k}\Bigr).
	\]
	This proves \eqref{eq:exp-upper} with
	\[
		\gamma_\alpha:=\frac{\min\{\log 2,\ \alpha-\log 2\}}{2(\alpha+1)},
		\qquad
		M_\alpha:=1+\frac{1}{1-e^{-(\alpha-\log 2)}}.
	\]
\end{proof}

\begin{corollary}\label{cor:linear-necessity-sublinear}
	Let $\eta\in(0,1)$ and let $\kappa:\Z_{>0}\to(0,\infty)$ satisfy $\kappa(m)=o(m)$.
	Then there exists an integer $m_0\ge 1$ such that for every integer $m\ge m_0$ there exist integers $n,d\ge 1$ with $n+d=m$ and a convex body $C\subset\R^{n+d}$ such that, writing $S:=C\cap(\Z^n\times\R^d)$, the projection $\proj_{\R^n}(C)$ contains a translate of $\kappa(m)\mathbb{B}_\infty^n$ but
	\[
		\sup_{\y\in S} h_S(\y)\ \le\ \eta\,\H_d(S).
	\]
	Equivalently, there is no sublinear function $\kappa(m)=o(m)$ that can serve as a threshold based only on the dimension that guarantees a dimension-free constant fraction lower bound for mixed-integer halfspace depth.
\end{corollary}
\begin{proof}
	Choose any $\alpha>\log 2$ and let $\gamma_\alpha,M_\alpha$ be the constants from Theorem~\ref{thm:linear-necessity-total-dim}.
	Since $\kappa(m)=o(m)$, we have $m/\kappa(m)\to \infty$.
	Thus, there exists an integer $m_0\ge 1$ such that for every integer $m\ge m_0$, the integer $k:=\max\{2,\lceil \kappa(m)\rceil\}$ satisfies $m\ge 2(1+\alpha k)$ and $M_\alpha\exp(-\gamma_\alpha m/k)\le \eta$.
	Apply Theorem~\ref{thm:linear-necessity-total-dim} with this $(m,k)$ to obtain integers $n,d\ge 1$ with $n+d=m$ and a convex body $C\subset\R^{n+d}$ such that $\proj_{\R^n}(C)$ contains a translate of $k\mathbb{B}_\infty^n$.
	Since $\kappa(m)\le k$, the projection also contains a translate of $\kappa(m)\mathbb{B}_\infty^n$.
	Finally, \eqref{eq:exp-upper} implies $\sup_{\y\in S} h_S(\y)/\H_d(S)\le \eta$.
\end{proof}

Together with Theorem~\ref{thm:main}, this shows that the dependence on the total dimension in the box-radius assumption cannot be reduced to $o(n+d)$ under this type of hypothesis.

\section{Conclusion}\label{sec:conclusions}

The main contribution of this paper is a linear threshold for Oertel's conjecture on mixed-integer halfspace depth. Under the assumption that the projection $K=\proj_{\R^n}(C)$ contains a translate of $k\mathbb{B}_\infty^n$ with $k\ge 3e(n+d)$, we have shown the existence of a point in $S=C\cap(\Z^n\times\R^d)$ achieving depth at least $\frac{1}{2e}\H_d(S)$, which implies the conjectured $\frac{1}{2^n e}$ bound. Our condition scales linearly in the total dimension and linearly in $d$, improving upon prior thresholds that depend exponentially on $n$ or polynomially as $d^2 n^{3/2}$. We have also shown that a linear threshold cannot be relaxed: no sublinear radius $\kappa(m)=o(m)$ can guarantee a dimension independent constant fraction lower bound.

The most natural direction is to investigate whether Conjecture~\ref{conj:oertel} holds for every convex body $C\subset\R^{n+d}$ without any largeness assumption on the projection. The authors do not see an obvious approach to this question and it seems to require fundamentally new ideas. There also seems to be scope for improving the quantitative bounds: Theorem~\ref{thm:main} guarantees depth $\frac{1}{e}-\Theta((n+d)/k)$ for large $k$, while Theorem~\ref{thm:linear-necessity-total-dim} constructs sets with depth $\exp(-\Theta(m/k))$. Closing this gap and determining the exact asymptotic dependence on $(n,d,k)$ remain open.

Finally, the point $\y^\star$ in Theorem~\ref{thm:main} arises from centroid rounding and its efficient computation is unclear. Can one approximate such a point from standard oracles for $C$? More broadly, can one design a polynomial time cutting plane procedure that provably removes a constant fraction of mixed-integer volume per iteration? Progress on these questions would have direct implications for the complexity theory of mixed-integer convex optimization.

\renewcommand{\ackname}{Acknowledgments and Disclosure of Funding}
\begin{acknowledgements}
Both authors gratefully acknowledge support from the Air Force Office of Scientific Research (AFOSR) grant FA9550-25-1-0038. The first author also received support from a MINDS Fellowship awarded by the Mathematical Institute for Data Science (MINDS) at Johns Hopkins University.
\end{acknowledgements}

\bibliographystyle{spbasic}
\bibliography{references}

\end{document}